\documentclass[twoside]{article}
\usepackage{amsfonts, amsmath}
\usepackage{xcolor}
\usepackage{amssymb}
\usepackage{amsfonts}
\usepackage{microtype}
\usepackage{enumitem}
\usepackage{amsmath,amssymb,amsfonts}%
\usepackage{amsthm}%
\usepackage{mathrsfs}%
\usepackage[colorlinks=true]{hyperref}
\hypersetup{linkcolor=blue}
\usepackage[dvips, lmargin=2.5cm, rmargin=1.5cm, tmargin=2.5cm, bmargin=2cm]{geometry}
\DeclareMathOperator{\rank}{rank } 
\newtheorem{theorem}{Theorem}
\newtheorem{lemma}{Lemma}
\newtheorem{proposition}[theorem]{Proposition}

\newtheorem{example}{Example}%
\newtheorem{remark}{Remark}%

\newtheorem{definition}{Definition}

\newcommand{\verti}[1]{{\left\vert\kern-0.25ex\left\vert\kern-0.25ex\left\vert #1
\right\vert\kern-0.25ex\right\vert\kern-0.25ex\right\vert}}
\newcommand{\vertiii}[1]{{\left\vert\kern-0.25ex\left\vert\kern-0.25ex\left\vert #1
    \right\vert\kern-0.25ex\right\vert\kern-0.25ex\right\vert}}
\title{
\textbf{Singular fractional double-phase  problems with variable exponent  via Morse's theory}}

\numberwithin{equation}{section}

\date{}
\author{A. Aberqi$^{1}$, A. Ouaziz$^{2}$ \\
$^{1}$Sidi Mohamed ben abdellah university, National School of Applied Sciences\\, Fez, Morocco\\
\href{www.google.com}{\color{blue}ahmed.aberqi@usmba.ac.ma}
\\
$^{2}$Faculty of Sciences Dhar El Mahraz, Sidi Mohamed Ben Abdellah University,\\ Fez Morocco.
$^{2}$  \href{http//:www.gmail.com}{\color{blue}abdesslam.ouaziz1994@gmail.com}$^{2}$}

\everymath{\displaystyle}
\begin{document}
\maketitle
\vspace{0.5cm}
%%%%%%%%%%%%%%%%%%%%%%%%%%%%%%%%%%%%%%%%%%%%%%%%%%%%%%%%%%%%%%%%%%%%
%                              abstract                            %
%%%%%%%%%%%%%%%%%%%%%%%%%%%%%%%%%%%%%%%%%%%%%%%%%%%%%%%%%%%%%%%%%%%%
\begin{abstract}
	In this manuscript, we deal with a class of fractional non-local problems involving a singular term and  vanishing potential of the form:
\begin{eqnarray*}
\begin{gathered}
\left\{\begin{array}{llll}
 \mathcal{L}^{s_{1}, s_{2}}_{p(\mathrm{x}, .), q(\mathrm{x}, .)}\mathrm{w}(\mathrm{x})&= \displaystyle\frac{g(\mathrm{x}, \mathrm{w}(\mathrm{x}))}{  \mathrm{w}(\mathrm{x})^{\xi(\mathrm{x})}}  + \mathcal{V}(\mathrm{x}) \vert \mathrm{w}(\mathrm{x}) \vert^{\sigma(\mathrm{x})-2}  \mathrm{w}(\mathrm{x})    & \text { in } & \mathcal{U}, \\
\hspace{2cm}  \mathrm{w}&> 0 & \text { in }&  \mathcal{U},\\ 
\hspace{2cm}  \mathrm{w}&=0 & \text { in }&  \mathbb{R}^{N} \backslash \mathcal{U},
\end{array}\right.
\end{gathered}
\end{eqnarray*}
where, $ \mathcal{L}^{s_{1}, s_{2}}_{p(\mathrm{x}, .), q(\mathrm{x}, .)}$  is a $\left(p(\mathrm{x}, .), q(\mathrm{x}, .)\right)-$ fractional double-phase operator with $ s_{1},s_{2 }\in \left( 0, 1\right)$, $g,$ and $\mathcal{V}$ are functions that satisfy some conditions. The strategy of the proof for these results is to approach the problem proximatively and calculate the critical groups. Moreover, using Morse's theory to prove our problem has infinitely many solutions.
\end{abstract}
%%%%%%%%%%%%%%%%%%%%%%%%%%%%%%%%%%%%%%%%%%%%%%%%%%%%%%%%%%%%%%%%%%%%
%                            Key Words:                            %
%%%%%%%%%%%%%%%%%%%%%%%%%%%%%%%%%%%%%%%%%%%%%%%%%%%%%%%%%%%%%%%%%%%%
{\bf Key words:}  Fractional Double phase, singular linearity, Variable exponent, Morse theory, critical groups. \\
%%%%%%%%%%%%%%%%%%%%%%%%%%%%%%%%%%%%%%%%%%%%%%%%%%%%%%%%%%%%%%%%%%%%
%                 AMS Subject Classifications:                     %
%%%%%%%%%%%%%%%%%%%%%%%%%%%%%%%%%%%%%%%%%%%%%%%%%%%%%%%%%%%%%%%%%%%%
{\bf AMS:}	35J60,  58J05, 35R11, 35J75, 35J60, 46E35.
%%%%%%%%%%%%%%%%%%%%%%%%%%%%%%%%%%%%%%%%%%%%%%%%%%%%%%%%%%%%%%%%%%%%
%              Section 1 : Introduction                            %
%%%%%%%%%%%%%%%%%%%%%%%%%%%%%%%%%%%%%%%%%%%%%%%%%%%%%%%%%%%%%%%%%%%%
\section{Introduction}
Marston Morse, a mathematician, created Morse's hypothesis in the 1920s. He was a member of the university at the Institute for Advanced Study, and Princeton University released Topological Methods in the Theory of Functions of a Complex Variable in 1947 as part of the Annals of Mathematics Studies series. A well-known publication by theoretical physicist Edward Witten that connects Morse's theory to quantum field theory has garnered a lot of attention for this idea during the past two decades.
Morse's theory and computation of critical groups are useful tools for studying the multiplicity and existence of solutions to nonlinear problems. As far as we know, this method is rarely used in the study of problems of differential equations.  For this, we'll provide a brief overview of the technique.\\  Let  $\displaystyle W$ be a real  Banach space, $\displaystyle \phi\in C^{1}(W, \mathbb{R})$ satisfies the Palais-Smale  condition, and  $\displaystyle c\in \mathbb{R}.$ We consider  the following sets: $$\phi^{c}=\left\lbrace  \mathrm{u}\in W:  \phi(\mathrm{u})\leq c \right\rbrace ,$$
and $$K_{\phi}= \left\lbrace  \mathrm{u}\in W:  \phi^{'}(\mathrm{u})=0 \right\rbrace . $$ The critical groups of $\phi$
 at $\mathrm{u}$ are defined by  $$C_{k}(\phi, \mathrm{u})= H_{k}\big(\phi^{c} \cap U, \phi^{c} \cap U \backslash  \{ \mathrm{u} \}\big),$$
 where $\displaystyle k\in\mathbb{N}, $ $\displaystyle U$ is a neighbourhood of  $\displaystyle \mathrm{u}$ such that $\displaystyle  K_{\phi}\cap U= \{\mathrm{u} \},$ and   $\displaystyle H_{k}$ is the singular relative homology with coefficient in an Abelian group $\displaystyle G$, see  \cite{papageorgiou} for more details. Moreover, the authors in \cite{bartsch} introduced  the critical groups of $\phi$ at infinity by $$C_{k}(\phi, \infty)= H_{k}(W, \phi^{a}),$$ where $\displaystyle a$ is less than all critical values and $\displaystyle k\in\mathbb{N}.$  Concerning the connection between critical groups and critical points of  $\displaystyle \phi.$ Wellem et al. in \cite{willem} proved the following statement:
\begin{enumerate}
\item[1)] If $\displaystyle C_{k}(\phi, \infty) \ncong 0$ for some $\displaystyle k \in \mathbb{N}$, then $\displaystyle \phi$ has critical point $\displaystyle \mathrm{u}$ and satisfies $\displaystyle C_{k}(\phi, \mathrm{u}) \ncong 0$.
\item[2)]  Let $\displaystyle \theta \in W$ be an isolated critical point of $\displaystyle \phi$. If $\displaystyle C_{k}(\phi, \infty) \ncong C_{k}(\phi, \theta)$ for some $\displaystyle k \in \mathbb{N}$, then $\displaystyle \phi$ must have a non-zero critical point.
\end{enumerate}
Readers may refer to \cite{ bartsch, eilenberg, liu,  perera, willem} and the references therein for further insights and details on algebraic topology,  Morse's theory, and critical groups.\\
In recent years, double-phase differential operators have garnered significant interest among researchers, owing to their versatile applications across various scientific domains, with a particular focus on their relevance in physical processes. To illustrate, Zhikov \cite{Zhikov} proved in the context of elasticity theory that the modulation coefficient $\mu(.)$ plays a pivotal role in shaping the geometry of composites composed of two distinct materials characterized by different curing exponents, namely $p$ and $q.$\\
To set the stage for our motivation, we first provide a brief overview of prior research. In his work, Zhikov \cite{Zhikov1} introduced and examined functionals characterized by integrands that exhibit varying ellipticity depending on the location, thus offering models for strongly anisotropic materials. As an illustrative example, he employed the following function as a prototype:
\begin{equation}\label{functional}
 \mathrm{w} \mapsto \int_{\mathcal{U}} \left( \vert \nabla \mathrm{w} \vert^p + \mu(\mathrm{x})\vert \nabla \mathrm{w} \vert^q\right) d\mathrm{x}.
    \end{equation}
Following this, multiple research endeavours were undertaken in this particular direction, with notable mentions including the influential contributions of Baroni et al. in  \cite{Baroni,  Baroni1}. For further findings, readers are encouraged to consult the references provided in  \cite{aberqi7,  Crespo-Blanco, Gasinski, Papageorgiou1}. \\
 The primary focus of our present paper is to investigate the non-local version  of double phase function type \eqref{functional}, for variable exponents $p(x,)$ and $q(x,)$ and fractional constant orders $0 < s_1,s_2 < 1, $ of the form: \begin{align}\label{operator}
\mathcal{L}^{s_{1}, s_{2}}_{p(\mathrm{x}, .), q(\mathrm{x}, .)}\mathrm{w}(\mathrm{x})=2 \lim _{\varepsilon \rightarrow 0^{+}} \int_{\mathcal{U} \backslash \mathfrak{B}_{\varepsilon}(\mathrm{x})}    \left[ \frac{ \vert \mathrm{w}(\mathrm{x})-\mathrm{w}(\mathrm{y})\vert^{p(\mathrm{x}, \mathrm{y})-2}}{ \vert \mathrm{x}-\mathrm{y}\vert^{N+s_{1} p(\mathrm{x}, \mathrm{y})}}+ \frac{ \vert \mathrm{w}(\mathrm{x})-\mathrm{w}(\mathrm{y})\vert^{q(\mathrm{x}, \mathrm{y})-2}}{ \vert \mathrm{x}-\mathrm{y}\vert^{N+s_{2} q(\mathrm{x}, \mathrm{y})}} \right]  (\mathrm{w}(\mathrm{x})-\mathrm{w}(\mathrm{y}))d\mathrm{y},    
 \end{align} 
 where $\mathfrak{B}_{\varepsilon}(\mathrm{x})$ is the ball of 
 $\mathcal{U}$  of radius $\varepsilon $ and center $\mathrm{x}.$
By studying   the following  class of variable-order fractional of double-phase problems driven by $\displaystyle (p(\mathrm{x}, .), q(\mathrm{x}, .))-$ fractional Laplacian with variable exponents involving a singular term and vanishing potential:
\begin{eqnarray}\label{k1}
\begin{gathered}
\left\{\begin{array}{llll}
 \mathcal{L}^{s_{1}, s_{2}}_{p(\mathrm{x}, .), q(\mathrm{x}, .)}\mathrm{w}(\mathrm{x})&= \displaystyle\frac{g(\mathrm{x}, \mathrm{w}(\mathrm{x}))}{  \mathrm{w}(\mathrm{x})^{\xi(\mathrm{x})}}  + \mathcal{V}(\mathrm{x}) \vert \mathrm{w}(\mathrm{x}) \vert^{\sigma(\mathrm{x})-2}  \mathrm{w}(\mathrm{x})    & \text { in } & \mathcal{U}, \\
\hspace{2cm}  \mathrm{w}&> 0 & \text { in }&  \mathcal{U},\\ 
\hspace{2cm}  \mathrm{w}&=0 & \text { in }&  \mathbb{R}^{N} \backslash \mathcal{U}.
\end{array}\right.
\end{gathered}
\end{eqnarray}
Here $\mathcal{U}\subset \mathbb{R}^N$ an open bounded set, we start by fixing $s_{1},$ $ s_{2}\in (0,1),$   $p, q: \mathcal{U} \times \mathcal{U}  \to (1, \infty), $  $\sigma:  \mathcal{U}   \to (1, \infty),$  and   $\xi:  \mathcal{U}   \to (0, 1]$  are  continuous functions that satisfy the following conditions:
\begin{equation} \label{l20}
    p(\mathrm{x}- \mathrm{z},  \mathrm{y}- \mathrm{z})=  p(\mathrm{x},  \mathrm{y}),  \text { for all }  (\mathrm{x}, \mathrm{y}, \mathrm{x}) \in \mathcal{U}\times\mathcal{U}\times\mathcal{U},
\end{equation}
\begin{equation} \label{l3}
   p(\mathrm{x},  \mathrm{y})=  p(\mathrm{y},  \mathrm{x}),  \text { for all }  (\mathrm{x}, \mathrm{y}) \in \mathcal{U}\times\mathcal{U},
\end{equation}
\begin{equation} \label{l30}
1<\sigma^{-}<\sigma^{+}< q^{-}<q^{+}<p^{-}<p^{+}<+\infty,
\end{equation}
with   $\sigma^{-}=\min _{  \mathrm{x}\in\mathcal{U}}\sigma(\mathrm{x}), $ $\sigma^{+}=\max_{\mathrm{x}\in \mathcal{U}}\sigma(\mathrm{x}), $ $q^{-}=\min _{  (\mathrm{x}, \mathrm{y})\in\mathcal{U} \times \mathcal{U}}q(\mathrm{x},  \mathrm{y}), $ $q^{+}=\max_{ (\mathrm{x}, \mathrm{y})\in\mathcal{U} \times\mathcal{U}  }q(\mathrm{x}, \mathrm{y}), $  $p^{-}=\min _{  (\mathrm{x}, \mathrm{y})\in\mathcal{U} \times \mathcal{U}}p(\mathrm{x},  \mathrm{y}), $ $p^{+}=\max_{ (\mathrm{x}, \mathrm{y})\in\mathcal{U} \times\mathcal{U}  }p(\mathrm{x}, \mathrm{y}),$
$\mathcal{V}$
 vanishing potential satisfies the following assumptions:
 \begin{enumerate}
 \item [$( \mathrm{V}) $]      $\mathcal{V}:  \mathbb{R}^{N} \rightarrow \mathbb{R}$ is a continuous function,  there exist $\theta_{1}  >0,$  and $ 0<\eta_{1}<1$ such that $$\mathcal{V}(\mathrm{x})>\theta_{1} >0   \text { and }
 \int_{\mathbb{R}^{N}} \mathcal{V}(\mathrm{x})\vert \mathrm{w}(\mathrm{x}) \vert^{\sigma(\mathrm{x})} d \mathrm{x} \leq \eta_{1} \Vert  \mathrm{w}\Vert_{Y_{1}},   $$  for all $\mathrm{x} \in \mathbb{R}^{N}$, and $\mathrm{w}\in  Y_{1}$ with $Y_{1}$  is the fractional Sobolev space see section  \ref{Generalized fractional Sobolev space} for more details.
 \end{enumerate}
$\displaystyle g: \mathcal{U}\times\mathbb{R}\to\mathbb{R}$ is a Carath\'eodory function   that satisfies the  following condition: \\
$( \mathcal{H}_{1})$ There exist $\displaystyle \beta \in L^{\infty}(\mathcal{U}),$ and a continuous   function $\displaystyle r:\mathcal{U}  \to  (1, +\infty)$ such that 
  $$\displaystyle 1< r(\mathrm{x})< p^{\star}_{s_{1}}(\mathrm{x})=\frac{N p(\mathrm{x},\mathrm{x})}{N- s_{1}p(\mathrm{x},\mathrm{x})},$$   and 
 $$g(\mathrm{x},  \mathrm{y}) \leq  \beta(\mathrm{x}) \left( 1+ \vert\mathrm{y}\vert^{r(\mathrm{x})-1}\right) ,   \hspace{1.5cm} \text { a.e.  }\,  x\in  \mathcal{U},  \mathrm{y}\in \mathbb{R},$$
$\mathcal{L}^{s_{1}, s_{2}}_{p(\mathrm{x}, .), q(\mathrm{x}, .)}$ 
 is the double-phase operator defined by \eqref{operator}.
  Using   Morse's theory,  local linking arguments, and variational analysis, more precisely, by computing the critical groups of the energy functional associated with the approximated equations by using some variational method combined with Morse's theory,  we prove the existence of infinitely many solutions to  problem\eqref{k1}.

%%%%%%%%%%%%%%%%%%%%%%%%%%%%%%%%%%%%%%%%%%%%%%%%%%%%%%%%%%

%%%%%%%%%%%%%%%%%%%%%%%%%%%%%%%%%%%%%%%%%%%%%%%%%%%%%%%%%%
 For the $p(\mathrm{x},.)$ Laplacian operator, the approaches for ensuring the existence of solutions were addressed in greater depth, we quote, the relevant work of Bahrouni and Radulescu \cite{Bahrouni} who developed some qualitative properties on the fractional Sobolev space  $W^{s,q(\mathrm{x}), p(\mathrm{x}, \mathrm{y})}(\mathcal{U})$ for $s \in(0,1)$ and $\mathcal{U}$
being a bounded domain in $\mathbb{R}^{n}$ with a Lipschitz boundary. Moreover, they studied the existence
of solutions to the following problem:
\begin{eqnarray}\label{frac}
\begin{gathered}
\left\{\begin{array}{llll}
\mathcal{L} \mathrm{w}(\mathrm{x})+\vert \mathrm{w}(\mathrm{x})\vert ^{q(\mathrm{x})-1} \mathrm{w}(\mathrm{x})&=\lambda\vert \mathrm{w}(\mathrm{x})\vert ^{r(\mathrm{x})-1} \mathrm{w}(\mathrm{x}) & \text { in } &\mathcal{U}, \\ 
\mathrm{w}&=0  &\text { in } &\partial \mathcal{U},
\end{array}\right.
\end{gathered}
\end{eqnarray}
where
$$
\mathcal{L} \mathrm{w}(\mathrm{x})=2 \lim _{\varepsilon \rightarrow 0^{+}} \int_{\mathcal{U} \backslash \mathfrak{B}_{\varepsilon}(\mathrm{x})} \int_{\mathcal{U}} \frac{\vert \mathrm{w}(\mathrm{x})-\mathrm{w}(\mathrm{y})\vert ^{p(\mathrm{x}, \mathrm{y})-2}(\mathrm{w}(\mathrm{x})-\mathrm{w}(\mathrm{y}))}{\vert \mathrm{x}-\mathrm{y}\vert^{n+s p(\mathrm{x}, \mathrm{y})}} d \mathrm{x},
$$
$\lambda>0,$ and $1<r(\mathrm{x})<p^{-}=\min _{(\mathrm{x}, \mathrm{y}) \in \mathcal{U} \times \mathcal{U}} p(\mathrm{x}, \mathrm{y})$.\\
 More recently,  authors in \cite{ Biswas}  studied the double phase version of problem\eqref{frac} with non-linearity logarithmic
$$ \mathcal{L}^{s_{1}, s_{2}}_{p(\mathrm{x}, .), q(\mathrm{x}, .)}\mathrm{w}(\mathrm{x})=\lambda\vert \mathrm{w}(\mathrm{x})\vert ^{r(\mathrm{x})-1} \mathrm{u}(\mathrm{x})+ \mu(x)\vert \mathrm{w}(\mathrm{x})\vert^{r(\mathrm{x})-2}\ln(\vert\mathrm{w}(\mathrm{x})\vert)  \text { in } \mathcal{U}.
$$
Readers may refer to  \cite{  Ayazoglu,    Bahrouni,   aberqi, fan,  liu1} and the references therein for more ideas and techniques developed to guarantee the existence of weak solutions for a class of nonlocal fractional problems with variable exponents. \\
The novelty of our work is to study the existence of infinitely many solutions to a class double phase problems driven by $\mathcal{L}^{s_{1}, s_{2}}_{p(\mathrm{x}, .), q(\mathrm{x}, .)}$  double-phase operator involving a singular nonlinearity and vanishing potential with variable exponent,  by computing the critical groups of the energy functional associated to the approximated equations by using some variational method combining with Morse's theory.\\ 
The structure of this article is as follows. In section \ref{preliminaries},  we briefly introduce certain homology theory concepts. We also give definitions and basic properties for Lebesgue spaces and fractional Sobolev spaces with variable exponent.
 In section \ref{computation}, we  suggest the approximated problem (\ref{pro}), and use the homological theory  to  
compute critical groups of the energy functional associated with the approximated problem (\ref{pro}).
In paragraph \ref{infinitly}, we will use  Morse's relation to show that the approximated problem  (\ref{pro}) admits infinitely non-trivial solutions. In the last section, we will prove our fundamental Theorem  \ref{theorem fundamental}. 
\section{Mathematical background}\label{preliminaries}
   \subsection{Generalized Lebesgue space } \label{rappel}
We consider the set:
 $$C^{+}(\bar{\mathcal{U}})=\left\lbrace  m: \bar{\mathcal{U}} \to  \mathbb{R}^{+}:  m  \text { is a continuous  function and  } 1< m^{-}< m(y)< m^{+} <+\infty   \right\rbrace,$$
 where 
 $~ \displaystyle m^{-}= \min _{y \in \bar{\mathcal{U}} } m(y), ~
 \displaystyle m^{+}= \max _{y \in \bar{\mathcal{U}} } m(y).$
 \begin{definition}(see \cite{fan})
Let $\displaystyle m\in  C^{+}(\bar{\mathcal{U}}). $    We define the  generalized Lebesgue  space $L^{m(y)}(\mathcal{U})$ as usual: $$L^{m(y)}(\mathcal{U})=\left\lbrace u:\mathcal{U}\to \mathbb{R}  \text { is  a measurable function }: \exists \lambda>0: \int_{\mathcal{U}} \vert \frac{u(y)}{\lambda} \vert ^{m(y) }  dx<\infty \right\rbrace.$$
We equip this space with the so-called Luxemburg norm defined as follows: $$\vert w\vert _{L^{m(y)}(\mathcal{U})}=\inf  \left\lbrace \xi >0:  \int_{\mathcal{U}} \vert \frac{w(y)}{\xi} \vert ^{m(y) }  dy\leq 1\right\rbrace.$$
 \end{definition}
 \begin{lemma}(see \cite{fan})
For every $\displaystyle w\in L^{m(y)}(\mathbb{R}^{N}),$  the following properties hold:
\begin{itemize}
\item[i)]  If  $\quad $  $\displaystyle \vert w\vert_{L^{m(y)}(\mathbb{R}^{N})}<1,$  then  $\displaystyle \vert w\vert^{m^{-}}_{L^{m(y)}(\mathbb{R}^{N})}\leq  \rho_{m(y)}(w) \leq  \vert w\vert^{m^{+}}_{L^{m(y)}(\mathbb{R}^{N})}.$
\item[ii)]  If $\quad$ $\displaystyle \vert w\vert_{L^{m(y)}(\mathbb{R}^{N})}>1,$ then  $\displaystyle \vert w\vert^{m^{+}}_{L^{m(y)}(\mathbb{R}^{N})}\leq  \rho_{m(y)}(w) \leq  \vert w\vert^{m^{-}}_{L^{q(y)}(\mathbb{R}^{N})}.$
\item[iii)]   $\displaystyle \vert w\vert _{L^{m(y)}(\mathbb{R}^{N})}<1,   =1,   >1$ \, if only if \, $\displaystyle \rho_{m(y)}(w)<1,   =1,   >1,$
\end{itemize}
where  $\rho_{m(y)}:L^{m(y)}(\mathbb{R}^{N}) \to  \mathbb{R}$  is the mapping defined as follows $$\rho_{m(y)}(w)=\int_{\mathbb{R}^{N}}\vert w(y)\vert ^{m(y)}dy.$$
 \end{lemma}
 \begin{proposition} \label{proposition}(see \cite{fan})
For every $w$ and  $\displaystyle w_{n} \in  L^{m(y)}(\mathbb{R}^{N}),$ the following statements are equivalent:
\begin{itemize}
\item[i)] $\displaystyle \lim_{n\to  +\infty}  \vert w_{n}-w\vert _{L^{m(y)}(\mathbb{R}^{N})}= 0,$
\item[ii )]  $\displaystyle \lim_{n\to  +\infty}  \rho_{m(y)}(w_{n}-w)=0,$
\item[iii )] $\displaystyle w_{n}\rightarrow w$  in measure on  $\mathbb{R}^{N}$  and $\displaystyle \lim_{n \to  +\infty} \rho_{m(y)}(w_{n}) - \rho_{m(y)}w)=0.$
\end{itemize}
 \end{proposition}
 \begin{lemma} (H\"older's  inequality, see \cite{fan}) For every $m \in  C^{+}(\mathbb{R}^{N}),$ 
 the following inequality  holds:   $$\displaystyle \vert\int_{\mathbb{R}^{N}}{v}(y)w(y)dy\vert \leq \left(\frac{1}{m^{-}}+ \frac{1}{{m^{'}}^{-}}\right)\vert v\vert _{L^{m(y)}(\mathbb{R}^{N})}\vert w\vert _{L^{{m^{'}}(y)}(\mathbb{R}^{N})},$$  for all $ ({v}, w) \in  L^{m(y)}(\mathbb{R}^{N})\times  L^{{m^{'}}(y)}(\mathbb{R}^{N}),$ where $\displaystyle \frac{1}{m(y)}+ \frac{1}{{m^{'}}(y)}=1.$
 \end{lemma}  
\subsection{  Generalized fractional Sobolev space }\label{Generalized fractional Sobolev space}
We start by fixing the fractional exponent $s\in(0,1).$ Let $\mathcal{U}$ be an open  bounded set of $ \mathbb{R}^{N},$ $m_1\in C^{+}(\mathcal{U}), $  and $ \displaystyle p:\bar{\mathcal{U}}\times\bar {\mathcal{U}}\to (1, \infty)$  is a continuous function  that satisfies the conditions  (\ref{l20})- (\ref{l30}). We introduce the generalized fractional Sobolev space $ \displaystyle W^{s, m_{1}(\mathrm{x}) , p\left( \mathrm{x}, \mathrm{y}\right) }\left( \mathcal{U}\right) $ as follows
$$ \displaystyle W^{s, m_{1}(\mathrm{x}), p(\mathrm{x}, \mathrm{y})}(\mathcal{U})= \displaystyle \left\lbrace  \mathrm{w}\in L^{m_{1}(\mathrm{x})}(\mathcal{U}):  \frac{\mathrm{w}(\mathrm{x})-\mathrm{w}(\mathrm{y})}{ \beta \vert \mathrm{x}-\mathrm{y}\vert ^{s+\frac{N}{p(\mathrm{x}, \mathrm{y})}}} \in L^{p(\mathrm{x}, \mathrm{y})}(\mathcal{U}\times\mathcal{U} ) \text { for some }  \beta>0 \right\rbrace.$$
Let $\displaystyle [\mathrm{w}]^{s, p(\mathrm{x}, \mathrm{y})}= \displaystyle \inf \left\lbrace \beta>0:    \int_{\mathcal{U}\times\mathcal{U}}  \frac{ \vert \mathrm{w}(\mathrm{x})-\mathrm{w}(\mathrm{y})\vert ^{p(\mathrm{x}, \mathrm{y})}}{\beta ^{p(\mathrm{x}, \mathrm{y})} \vert \mathrm{x}-\mathrm{y} \vert ^{N+sp(\mathrm{x}, \mathrm{y})}} d\mathrm{x} d\mathrm{y}<1 \right\rbrace $ 
be the corresponding variable exponent Gagliardo seminorm. We equip  the space $\displaystyle W^{s, m_{1}(\mathrm{x}), p(\mathrm{x}, \mathrm{y})}(\mathcal{U})$  with the norm $$ \displaystyle \Vert \mathrm{w}\Vert_{W^{s, m_{1}(\mathrm{x}), p(\mathrm{x}, \mathrm{y})}(\mathcal{U})}= \displaystyle [\mathrm{w}]^{s, p(\mathrm{x}, \mathrm{y})}+ \vert \mathrm{w}\vert_{m_{1}(\mathrm{x})},$$
where $\displaystyle (L^{m_{1}(\mathrm{x})}(\mathcal{U}), \vert.\vert_{m_{1}(\mathrm{x})})$ is the  generalized Lebesgue space.
\begin{lemma} (see \cite{Bahrouni}) \label{separability}
 Let $ \displaystyle \mathcal{U}\subset \mathbb{R}^{N}$ be a Lipschitz-bounded domain, $ \displaystyle p:\mathcal{U}    \times  \mathcal{U} \rightarrow (1, +\infty)$ be a continuous function that  satisfies conditions  (\ref{l20})-(\ref{l30}), and $ \displaystyle m_{1}\in C^{+}(\bar{\mathcal{U}})$. Then  $\displaystyle W^{s, m_{1}(\mathrm{x}), p(\mathrm{x}, \mathrm{y})}(\mathcal{U}) $ is a  separable, and reflexive Banach space.   
\end{lemma}
\begin{theorem} (see\cite{aberqi6, aberqi, kaufmann})\label{embedding}
 Let $ \displaystyle \mathcal{U}\subset \mathbb{R}^{N}$ be a Lipschitz-bounded domain,  $ \displaystyle p:\mathcal{U}\times\mathcal{U} \rightarrow (1, +\infty)$ be a continuous function   that satisfies conditions  (\ref{l20})-(\ref{l30}) $ \displaystyle m_{1}\in C^{+}(\mathcal{U}), $ and
   $$\displaystyle sp(\mathrm{x},\mathrm{y})<N, ~ p(\mathrm{x}, \mathrm{x})< m_{1}(\mathrm{x}),~\hbox{ for all} ~   (\mathrm{x}, \mathrm{y})\in\mathcal{U}^{2},$$
and $\displaystyle \mathfrak{\ell}:\overline{\mathcal{U}}  \rightarrow (1,+ \infty)$  is a continuous variable exponent such that $$ \displaystyle p^{*}_{s}(\mathrm{x})= \frac{Np(\mathrm{x},\mathrm{x})}{N-sp(\mathrm{x},\mathrm{x})}>\mathfrak{\ell}(\mathrm{x})\geq\mathfrak{\ell}^{-}=\min_{\mathrm{x} \in  \overline{\mathcal{U}}}\mathfrak{\ell}(\mathrm{x})>1.$$
  Then the space  $\displaystyle W^{s, m_{1}(\mathrm{x}), p(\mathrm{x}, \mathrm{y}}(\mathcal{U})$ is continuously embedded in $ \displaystyle L ^{\mathfrak{\ell}(y)}(\mathcal{U}).$ That is, there exists a positive constant $\displaystyle C= C(N,  s, p, m_{1}, \mathcal{U} )$ such that  $$\displaystyle \vert \mathrm{w}\vert _{L ^{{l}(\mathrm{x})}(\mathcal{U})}\leq  C\Vert \mathrm{w}\Vert_{W^{s, m_{1}(\mathrm{x}), p(\mathrm{x}, \mathrm{y})}(\mathcal{U})},\ \text{for all}\quad w\in W^{s, m_{1}(\mathrm{x}), p(\mathrm{x}, \mathrm{y})}(\mathcal{U}).$$  Moreover, this embedding is compact.
\end{theorem} 
\subsection{Homology theory }
We now present the fundamental tool that will be used to work with, namely the homology theory.
\begin{definition}\label{definition}( see \cite{perera})
Given Y is a Banach space, $\displaystyle \psi \in C(Y, \mathbb{R}),$ and 0 is an isolated critical point of $\displaystyle \psi $ such that $\displaystyle \psi(0)=0.$ Let $\displaystyle  m, n \in\mathbb{N}.$ We say that $\displaystyle \psi$  has a local $\displaystyle (m, n)-$ linking near the origin if there exist a neighbourhood  U of  0 and non-empty sets $\displaystyle F_{0}, $ $\displaystyle F\subset U,$ and  $\displaystyle D\subset Y $ such that $\displaystyle 0 \notin   F_{0} \subset F,$  $\displaystyle  F \cap D=\emptyset $ and 
\begin{enumerate}
\item[1)]  $
\displaystyle \left.\psi\right\vert_F \leq 0 <\left. \psi\right\vert_{U \cap D \backslash\{0\}},$ 
\item[2)] 0 is the only critical point of $\displaystyle \psi$ in $\displaystyle \psi^{0}\cap U,$   where $\displaystyle \psi^{0}=\{   \mathrm{w} \in Y: \psi(\mathrm{w})=0 \},$
\item[3)]  $\displaystyle \operatorname{Dim} i m\left(i^*\right)-\operatorname{Dim} i m\left(j^*\right) \geq n,$ where
$$
i^{*}: H_{m-1}\left(F_{0}\right) \rightarrow H_{m-1}(Y \backslash D) \text { and } j^{*}: H_{m-1}\left(F_{0}\right) \rightarrow H_{m-1}(F)
$$
are the homomorphisms induced by the inclusion maps $\displaystyle i:F_{0} \rightarrow Y \backslash D$ and $\displaystyle j: F_{0} \rightarrow F.$
\end{enumerate}
\end{definition}
\begin{lemma} (Morse's relation) (see \cite{papageorgiou}) \label{morse relation}
 If $Y$ is a Banach space, $\psi \in C^{1}(Y, \mathbb{R}), a, b \in \mathbb{R} \backslash \psi\left(\left\{K_{\psi}\right), a<b\right.$, $\psi^{-1}((a, b))$ contains a finite number of critical points $\left\{\mathrm{w}_{i}\right\}_{i=1}^{n}$ and $\psi$ satisfies  the Palais-Smale condition, then
 \begin{enumerate}
\item[1)] for all $k \in \mathbb{N}_0$, we have $ 
\sum_{i=1}^n \operatorname{rank} C_{k}\left(\psi, u_{i}\right) \geqslant \operatorname{rank} H_{k}\left(\psi^{b}, \psi^{a}\right)$;
\item[2)] if the Morse-type numbers $\sum_{i=1}^n \operatorname{rank} C_{k}\left(\psi, u_{i}\right)$ are finite for all $k \in \mathbb{N}_0$ and vanish for all large $k \in \mathbb{N}_{0},$ then so do the Betti numbers $\operatorname{rank} H_{k}\left(\psi^{b}, \psi^{a}\right)$ and we have
$$
\sum_{\mathrm{k} \geqslant 0} \sum_{i=1}^n \operatorname{rank} C_{k}\left(\psi, u_{i}\right) t^{k}=\sum_{\mathrm{k} \geqslant 0} \operatorname{rank} H_{k}\left(\psi^{b}, \psi^{a}\right) t^{k}+(1+t) Q(t) \text { for all } t \in \mathbb{R},
$$
where $Q(t)$ is a polynomial in $t \in \mathbb{R}$ with non-negative integer coefficients.
\end{enumerate}
\end{lemma}
\begin{theorem} (see \cite{Panda}
Let $\psi \in C^{2}(Y, \mathbb{R})$ satisfy the Palais-Smale condition, and let a be a regular value of $\psi$. Then, $H_*\left(Y, \psi^{a}\right) \neq 0,$ implies that $K_{\psi} \cap \psi^{a} \neq \emptyset.$
\end{theorem}
 \section{The Approximated Problem}
 We suggest an approximate problem sequence as       \begin{eqnarray}\label{pro}
\begin{gathered}
\left\{\begin{array}{llll}
 \mathcal{L}^{s_{1}, s_{2}}_{p(\mathrm{x}, .), q(\mathrm{x}, .)} \mathrm{w}_{n}(\mathrm{x})=  \displaystyle \frac{ g_{n}(\mathrm{x}, \mathrm{w}_{n}(\mathrm{x}))} {\left( \mathrm{w}_{n}(\mathrm{x})+\frac{1}{n}\right) ^{\xi (\mathrm{x})}}  + \mathcal{V}(\mathrm{x}) \vert \mathrm{w}_{n}(\mathrm{x})+\frac{1}{n} \vert^{\sigma(\mathrm{x})-2}  \left( \mathrm{w}_{n}(\mathrm{x})+\frac{1}{n}\right) & \text { in } & \mathcal{U}, \\
\hspace{4cm} \displaystyle \mathrm{w}_{n}>0 & \text { in }&  \mathcal{U},\\
\hspace{4cm} \displaystyle \mathrm{w}_{n}=0  &\text { in }&  \mathbb{R}^{N} \backslash \mathcal{U}, 
\end{array}\right.
\end{gathered}
\end{eqnarray} because the energy functional linked to our problem is not differentiable due to the inclusion of a singular term.
 $\displaystyle g_{n}(\mathrm{x}, t))=\displaystyle \min(n, g(\mathrm{x}, t)),$   $\displaystyle G_{n}( \mathrm{x},t)= \displaystyle\int_{0}^{t} \frac{g_{n}(\mathrm{x}, s)}{(s+ \frac{1}{n})^{\xi(\mathrm{x})}} ds, $ and  $g_{n}:\mathcal{U}\times\mathbb{R}\to \mathbb{R}$  is a sequence of functions that verifies the following conditions.\\
 $( \mathcal{H}_{2})$ There exist  $\displaystyle \theta>p^{+}$ and $\displaystyle r>0$ such that for a.e $\displaystyle \mathrm{x}\in  \mathcal{U} $ and $\displaystyle \vert\mathrm{x}\vert \geq r, $ $$\displaystyle 0< \theta G_{n}(x, t)\leq  \frac{t g_{n}(\mathrm{x}, t)}{(t+ \frac{1}{n})^{\xi(\mathrm{x})}}.$$
 $( \mathcal{H}_{3})$   It holds  $$ \lim_{t\to +\infty } \frac{g_{n}(\mathrm{x}, t)}{t^{p^{+}}}=l_{1}  \text{  uniformly for a.e }  \mathrm{x}\in  \mathcal{U},$$\\
 $( \mathcal{H}_{4})$  There exist $ \eta> \sigma^{-}$ and 
 $a_{3}>0 $  such that $$g_{n}(\mathrm{x}, t)t- \eta G_{n}(\mathrm{x}, t)\geq -a_{3}\vert t\vert^{p^{-}}$$
 for all $\mathrm{x}\in \mathcal{U}$ and $t\in\mathbb{R}.$
  \begin{example}
  Set $g_{n}(t)= l (t+\frac{1}{n})^{2},$ $p(\mathrm{x}, \mathrm{y})=p^{-}=2,$ and $\xi(\mathrm{x})=1.$ A trivial verification shows that $( \mathcal{H}_{1})-  ( \mathcal{H}_{4})$ are satisfied under a suitable condition on $\eta,$ $a_{3}, $ and $\theta.$
   \end{example}
 \begin{remark}
 If  the function $g$ satisfies  condition $( \mathcal{H}_{1}).$ Then, the sequence of function $g_{n}$ also verifies condition $( \mathcal{H}_{1}).$
 \end{remark}
 \subsection{Computation of   critical  group}\label{computation}
 For the sake of simplicity, we note  $Y_{1}:= W^{s_{1}, m_{1}(\mathrm{x}),  p(\mathrm{x},  \mathrm{y})}(\mathcal{U})$ and $Y_{2}:= W^{s_{2}, m_{2}(\mathrm{x}),  q(\mathrm{x},  \mathrm{y})}(\mathcal{U}).$
 \begin{definition}
 We  say that  $\displaystyle \{\mathrm{w}_{n}\}_{n \in \mathbb{N}} $ to be a weak solution of (\ref{pro}) if 
 \begin{align*}
 &\int_{\mathcal{U} \times \mathcal{U}} \frac{\vert\mathrm{w}_{n}(\mathrm{x})-\mathrm{w}_{n}(\mathrm{y})\vert^{p(\mathrm{x}, \mathrm{y})-2}(\mathrm{w}_{n}(\mathrm{x})-\mathrm{w}_{n}(\mathrm{y}))(\varphi(\mathrm{x})-\varphi(\mathrm{y}))}{ \vert \mathrm{x}-\mathrm{y}\vert^{N+s_{1} p(\mathrm{x}, \mathrm{y})} }d\mathrm{x}d\mathrm{y}\\
 &+\int_{\mathcal{U} \times \mathcal{U}} \frac{\vert\mathrm{w}_{n}(\mathrm{x})-\mathrm{w}_{n}(\mathrm{y})\vert^{q(\mathrm{x}, \mathrm{y})-2}(\mathrm{w}_{n}(\mathrm{x})-\mathrm{w}_{n}(\mathrm{y}))(\varphi(\mathrm{x})-\varphi(\mathrm{y}))}{ \vert \mathrm{x}-\mathrm{y}\vert^{N+s_{2} q(\mathrm{x}, \mathrm{y})} }d\mathrm{x}d\mathrm{y}\\&= \int_{\mathcal{U}} \left[ \frac{ g_{n}(\mathrm{x}, \mathrm{w}_{n}(\mathrm{x}))}{\left( \mathrm{w}_{n}(\mathrm{x})+\frac{1}{n}\right) ^{\xi (\mathrm{x})}}+ \mathcal{V}(\mathrm{x}) \vert \mathrm{w}_{n}(\mathrm{x})+\frac{1}{n} \vert^{\sigma(\mathrm{x})-2} \left( \mathrm{w}_{n}(\mathrm{x})+\frac{1}{n}\right) \right] 
  \varphi(\mathrm{x}) d\mathrm{x}, 
   \end{align*}
for all   $\varphi\in Y_{1}^{*},$  where $\displaystyle Y_{1}^{*}$ is the dual space of $\displaystyle Y_{1}.$
  \end{definition}
Consider the energy functional $\displaystyle \psi:Y_{1} \to \mathbb{R} $ defined by 
$$\psi(\mathrm{w}_{n})=\psi_{1}(\mathrm{w}_{n})-\psi_{2}(\mathrm{w}_{n})-  \psi_{3}(\mathrm{w}_{n}),$$
where 
\begin{align*}
\displaystyle \psi_{1}(\mathrm{w}_{n})&= \displaystyle \int_{\mathcal{U} \times \mathcal{U}}  \left[ \frac{1}{p(\mathrm{x}, \mathrm{y})} \frac{\vert\mathrm{w}_{n}(\mathrm{x})-\mathrm{w}_{n}(\mathrm{y})\vert^{p(\mathrm{x}, \mathrm{y})}}{ \vert \mathrm{x}-\mathrm{y}\vert^{N+s_{1} p(\mathrm{x}, \mathrm{y})} }+ \frac{1}{q(\mathrm{x}, \mathrm{y})} \frac{\vert\mathrm{w}_{n}(\mathrm{x})-\mathrm{w}_{n}(\mathrm{y})\vert^{q(\mathrm{x}, \mathrm{y})}}{ \vert \mathrm{x}-\mathrm{y}\vert^{N+s_{2} q(\mathrm{x}, \mathrm{y})} } \right] d\mathrm{x}d\mathrm{y},
\end{align*}
$
\displaystyle \psi_{3}(\mathrm{w}_{n})= \int_{\mathcal{U}}\frac{ \mathcal{V}(\mathrm{x})}{\sigma(\mathrm{x})}  \vert \mathrm{w}_{n}(\mathrm{x})+\frac{1}{n} \vert^{\sigma(\mathrm{x})}   d\mathrm{x}, 
$
  $\displaystyle \psi_{2}(\mathrm{w}_{n})=\displaystyle \int_{\mathcal{U}} G_{n}(\mathrm{x}, \mathrm{w}_{n}(\mathrm{x}))d\mathrm{x}, $ and $\displaystyle G_{n}(\mathrm{x}, t)=\int_{0 }^{t}\frac{g_{n}(\mathrm{x}, s)}{(s+\frac{1}{n})^{\xi(\mathrm{x})}}ds $ is the primitive of $\frac{g_{n}(\mathrm{x}, s)}{(s+\frac{1}{n})^{\xi(\mathrm{x})}}.$
\begin{lemma}\label{frechet}
If $\displaystyle g $ satisfies  $( \mathcal{H}_{1})$  condition and the potential $\mathcal{V}$ satisfies $( \mathrm{V}).$
Then $\displaystyle \psi_{2} + \psi_{3}\in C^{1}(Y_{1}, \mathbb{R}) $ and $$\langle (\psi_{2} + \psi_{3})^{'}(\mathrm{w}_{n}), \mathrm{v}_{n}\rangle =\int_{\mathcal{U}}  \left[  \frac{g_{n}\left( \mathrm{x},  \mathrm{w}_{n}(\mathrm{x})\right) }{(\mathrm{w}_{n}(\mathrm{x})+\frac{1}{n})^{\xi (\mathrm{x})}}+ \mathcal{V}(\mathrm{x}) \vert \mathrm{w}_{n}(\mathrm{x})+\frac{1}{n} \vert^{\sigma(\mathrm{x})-2}   \left( \mathrm{w}_{n}(\mathrm{x})+\frac{1}{n} \right) \right]  \mathrm{v}_{n}(\mathrm{x}) d\mathrm{x}, $$
 for all $\displaystyle \mathrm{w}_{n}, \mathrm{v}_{n} \in  Y_{1}.$
\end{lemma}
\begin{proof}
(i) $\displaystyle \psi_{2}$ is  Gateaux differentiable in $\displaystyle Y_{1}.$\\
 Let $\displaystyle \mathrm{w}_{n}, \mathrm{v}_{n} \in  Y_{1}, $ and $0<t<1, $ we have 
\begin{align*}
\frac{1}{t}(G_{n}\left( \mathrm{x}, \mathrm{w}_{n}+t\mathrm{v}_{n})-G_{n}(\mathrm{x}, \mathrm{w}_{n})\right) &=\frac{1}{t} \int _{0}^{\mathrm{w}_{n}+t\mathrm{v}_{n}} \frac{g_{n}(\mathrm{x}, s)}{(s+\frac{1}{n})^{\xi(\mathrm{x})}} ds- \frac{1}{t} \int _{0}^{\mathrm{w}_{n}} \frac{g_{n}(\mathrm{x}, s)}{(s+\frac{1}{n})^{\xi(\mathrm{x})}}ds\\
&=\frac{1}{t} \int _{\mathrm{w}_{n}}^{\mathrm{w}_{n}+t\mathrm{v}_{n}} \frac{g_{n}(\mathrm{x}, s)}{(s+\frac{1}{n})^{\xi(\mathrm{x})}}ds.
\end{align*}
By the mean value theorem, there exists $\displaystyle 0<\delta<1$ such that $$\frac{1}{t}(G_{n}(\mathrm{x}, \mathrm{w}_{n}+t\mathrm{v}_{n})-G_{n}(\mathrm{x}, \mathrm{w}_{n}))=\frac{g_{n}(x, \mathrm{w}_{n}+\delta t\mathrm{v}_{n})}{(\mathrm{w}_{n}+\delta t\mathrm{v}_{n}+\frac{1}{n})^{\xi (\mathrm{x})}}\mathrm{v}_{n}. $$
Combining   $( \mathcal{H}_{1})$ with Young's inequality, we have 
\begin{align*}
g_{n}(\mathrm{x}, \mathrm{w}_{n}+\delta t\mathrm{v}_{n})& \leq  g(\mathrm{x}, \mathrm{w}_{n}+\delta t\mathrm{v}_{n}) \\
 &\leq \beta (\vert\mathrm{v}_{n}\vert+  \vert\mathrm{w}_{n}+ \delta t \mathrm{v}_{n}\vert^{r(\mathrm{x})} \vert\mathrm{v}_{n}\vert)\\
 &\leq   \beta 2^{r^{+}} (1+  \vert\mathrm{w}_{n}\vert^{r(\mathrm{x})}+\vert\mathrm{v}_{n}\vert^{r(\mathrm{x})}).
\end{align*}
Since $\displaystyle r(\mathrm{x})\in (1,  p_{s_{1}}^{*}(\mathrm{x})),$  we  have $\displaystyle \mathrm{w}_{n}, \mathrm{v}_{n}\in  L^{r(\mathrm{x}}(\mathcal{U}).$ Thanks to the Lebesgue's  dominated converge Theorem, we get 
\begin{align}
\begin{split}
\displaystyle \lim_{t \to 0} \frac{1}{t} (G_{n}(\mathrm{x}, \mathrm{w}_{n}+t\mathrm{v}_{n})-G_{n}(\mathrm{x}, \mathrm{w}_{n}))&=  \lim_{t\to 0} \int_{\mathcal{U}} \frac{g_{n}(\mathrm{x},\mathrm{w}_{n}+ \delta t\mathrm{v}_{n} )}{(\mathrm{w}_{n}+ \delta t \mathrm{v}_{n}+\frac{1}{n})^{\xi (\mathrm{x})}}  \mathrm{v}_{n} d\mathrm{x}\\
&= \int_{\mathcal{U}} lim_{t\to 0} \frac{g_{n}(\mathrm{x},\mathrm{w}_{n}+ \delta t\mathrm{v}_{n} )}{(\mathrm{w}_{n}+ \delta t \mathrm{v}_{n}+\frac{1}{n})^{\xi (\mathrm{x})}}  \mathrm{v}_{n} d\mathrm{x}\\
&= \int_{\mathcal{U}} \frac{g_{n}(\mathrm{x},\mathrm{w}_{n} )}{(\mathrm{w}_{n}+\frac{1}{n})^{\xi (\mathrm{x})}}  \mathrm{v}_{n} d\mathrm{x}.
\end{split}
\end{align}
\begin{align}\label{equation1}
\begin{split}
  \langle\psi^{'}_{3} (\mathrm{w}_{n}),  \mathrm{v}_{n}\rangle&= \lim_{t\to 0} \frac{\psi_{3} (\mathrm{w}_{n}+t  \mathrm{v}_{n} ) - \psi_{3} (\mathrm{w}_{n}) }{t}\\
  &= \lim_{t\to 0}  \int_{\mathcal{U}}  \frac{\mathcal{V}(\mathrm{x})} { t \sigma(\mathrm{x})} \left(  \vert \mathrm{w}_{n}+ \mathrm{v}_{n}t+\frac{1}{n} \vert ^{\sigma(\mathrm{x})}-  \vert \mathrm{w}_{n}+ \frac{1}{n} \vert ^{\sigma(\mathrm{x})}\right) d \mathrm{x}.
  \end{split}
\end{align}
Considering the function defined by $L:[0, 1]\to\mathbb{R}$ as  $L(z)=\frac{\mathcal{V}(\mathrm{x})}{ \sigma(\mathrm{x})} \vert \mathrm{w}_{n}+ z\mathrm{v}_{n}t+\frac{1}{n} \vert ^{\sigma(\mathrm{x})}.$ 
According to the mean value Theorem, there exists $0<\varepsilon<1$ such that 
 \begin{equation}\label{equation2}
L^{'}(z)(\varepsilon)= L(1)- L(0).
\end{equation}
Combining (\ref{equation1}) with (\ref{equation2}), it follows that  $\langle\psi^{'}_{3} (\mathrm{w}_{n}),  \mathrm{v}_{n}\rangle=\int_{\mathcal{U}}   \mathcal{V}(\mathrm{x}) \vert \mathrm{w}_{n}(\mathrm{x})+\frac{1}{n} \vert^{\sigma(\mathrm{x})-2} \left(  \mathrm{w}_{n}(\mathrm{x})+\frac{1}{n} \right)   \mathrm{v}_{n}(\mathrm{x}) d\mathrm{x}. $
\end{proof}
(ii) The continuity of Gateaux-derivatives. Let  $\displaystyle 
\{\mathrm{w}_{n, k} \}_{k\in \mathbb{N} }\subset  Y_{1}  $ such that $\displaystyle \mathrm{w}_{n, k} \to \mathrm{w}_{n}  $ strongly in   $\displaystyle Y_{1}$ as $k \to +\infty.$ We use H\"{o}lder's inequality and condition $( \mathcal{H}_{1}), $ we have that  
\begin{align*}
\displaystyle \int_{\mathcal{U}}  \vert\frac{g_{n}(\mathrm{x}, \mathrm{w}_{n, k})}{(\mathrm{w}_{n, k}+\frac{1}{n})^{\xi(\mathrm{x})}}\vert^{ r^{'}(\mathrm{x})} d\mathrm{x}& \leq  \int_{\mathcal{U}} \vert g_{n}(\mathrm{x}, \mathrm{w}_{n, k})\vert^{ r^{'}(\mathrm{x})} d\mathrm{x}\\
& \leq
\int_{\mathcal{U}} \vert g(\mathrm{x}, \mathrm{w}_{n, k})\vert^{ r^{'}(\mathrm{x})} d\mathrm{x}\\
& \displaystyle \leq 2 ^{\frac{r^{+}+1}{r^{+}-1}} \|\beta \|_{\infty}^{\frac{r^{+}+1}{r^{+}-1}} \int_{\mathcal{U}} \vert\mathrm{w}_{n, k}\vert^{r(\mathrm{x})} d\mathrm{x}\\
& \leq C(\beta,  r^{+})\int_{\mathcal{U}} \vert\mathrm{w}_{n, k}\vert^{r(\mathrm{x})} d\mathrm{x} \\
& \leq C(\|\beta \|_{\infty},  r^{+}) \| \vert \mathrm{w}_{n, k}\vert \|_{L ^{\frac{p_{s_{1}^{*}(\mathrm{x})}}{r(\mathrm{x})}}(\mathcal{U})} \| 1\|_{L ^{\frac{p_{s_{1}^{*}(\mathrm{x})}}{p_{s_{1}^{*}(\mathrm{x})}-r(\mathrm{x})}}(\mathcal{U})}. 
\end{align*}
So, the sequence  $\displaystyle \{ \vert  \frac{g_{n}(\mathrm{x}, \mathrm{w}_{n, k})}{(\mathrm{w}_{n, k}+\frac{1}{n})^{\xi(\mathrm{x})}}-  \frac{g_{n}(\mathrm{x}, \mathrm{w}_{n})}{(\mathrm{w}_{n}+\frac{1}{n})^{\xi(\mathrm{x})}} \vert^{r(\mathrm{x})}\}_{k\in \mathbb{N}} $ is uniformly bounded and equi-integrable in $\displaystyle L^{1}(\mathcal{U}).$ Thanks to Vitali converge theorem implies $$\lim_{k \to +\infty}\int_{\mathcal{U}} \vert  \frac{g_{n}(\mathrm{x}, \mathrm{w}_{n, k})}{(\mathrm{w}_{n, k}+\frac{1}{n})^{\xi(\mathrm{x})}}-  \frac{g_{n}(\mathrm{x}, \mathrm{w}_{n})}{(\mathrm{w}_{n}+\frac{1}{n})^{\xi(\mathrm{x})}} \vert^{r^{'}(\mathrm{x})} d\mathrm{x}=0, $$
where  $\displaystyle \frac{1}{r^{'}(\mathrm{x})}+ \frac{1}{r(\mathrm{x})}=1.$ Thus, by Theorem  \ref{embedding} and H\"{o}lder's inequality, we have 
\begin{align*}
\begin{split}
\displaystyle \| \psi^{'}_{2} (\mathrm{w}_{n, k})- \psi^{'}_{2} (\mathrm{w}_{n})\|_{Y_{1}^{*}}
&=\displaystyle\sup_{\mathrm{v}_{n}\in Y_{1}}   \| \langle\psi^{'}_{2} (\mathrm{w}_{n, k})- \psi^{'}_{2} (\mathrm{w}_{n}),  \mathrm{v}_{n} \rangle
\|_{Y_{1}}\\
&\displaystyle \leq     \vert  \langle\psi^{'}_{2} (\mathrm{w}_{n, k})- \psi^{'}_{2} (\mathrm{w}_{n}) ,  \mathrm{v}_{n}\rangle \vert\\
& \displaystyle \leq \| \frac{g_{n}(\mathrm{x}, \mathrm{w}_{n, k})}{(\mathrm{w}_{n, k}+\frac{1}{n})^{\xi(\mathrm{x})}}-  \frac{g_{n}(\mathrm{x}, \mathrm{w}_{n})}{(\mathrm{w}_{n}+\frac{1}{n})^{\xi(\mathrm{x})}} \|_{L^{q^{'}_{1}(\mathrm{x})}(\mathcal{U})} \| \mathrm{v}_{n}\|_{L^{q_{1}(\mathrm{x})}(\mathcal{U})}  \to 0 \text{ as } k\to +\infty, 
\end{split}
\end{align*}
where $\displaystyle Y_{1}^{*}$ is the dual space of $\displaystyle Y_{1}.$ Similarly, we prove that $\psi^{'}_{3}$ continuous in $Y_{1}.$
From the Lemma \ref{frechet} and  Lemma 4.1 in \cite{aberqi},  we have that  $\displaystyle \psi\in C^{1}(Y_{1}, \mathbb{R}), $ and
\begin{align*} 
\langle\psi^{'} (\mathrm{w}_{n, k}),  \mathrm{v}_{n}\rangle = &\int_{\mathcal{U}\times\mathcal{U}} \frac{\vert\mathrm{w}_{n, k}(\mathrm{x})-\mathrm{w}_{n, k}(\mathrm{y})\vert^{p(\mathrm{x}, \mathrm{y})-2}(\mathrm{w}_{n, k}(\mathrm{x})-\mathrm{w}_{n, k}(\mathrm{y}))(\mathrm{v}_{n}(\mathrm{x})-\mathrm{v}_{n}(\mathrm{y}))}{ \vert \mathrm{x}-\mathrm{y}\vert^{N+s_{1} p(\mathrm{x}, \mathrm{y})} }d\mathrm{x}d\mathrm{y}\\
&+\int_{\mathcal{U} \times \mathcal{U}} \frac{\vert\mathrm{w}_{n, k}(\mathrm{x})-\mathrm{w}_{n, k}(\mathrm{y})\vert^{q(\mathrm{x}, \mathrm{y})-2}(\mathrm{w}_{n, k}(\mathrm{x})-\mathrm{w}_{n, k}(\mathrm{y}))(\mathrm{v}_{n}(\mathrm{x})-\mathrm{v}_{n}(\mathrm{y}))}{ \vert \mathrm{x}-\mathrm{y}\vert^{N+s_{2} q(\mathrm{x}, \mathrm{y})} }d\mathrm{x}d\mathrm{y} \\
 &-\int_{\mathcal{U}}  \left[  \frac{g_{n}\left( \mathrm{x},  \mathrm{w}_{n}(\mathrm{x})\right) }{(\mathrm{w}_{n}(\mathrm{x})+\frac{1}{n})^{\xi (\mathrm{x})}}+ \mathcal{V}(\mathrm{x}) \vert \mathrm{w}_{n}(\mathrm{x})+\frac{1}{n} \vert^{\sigma(\mathrm{x})-2}   \left( \mathrm{w}_{n}(\mathrm{x})+\frac{1}{n} \right) \right]  \mathrm{v}_{n}(\mathrm{x}) d\mathrm{x}, 
\end{align*}
for all $\mathrm{v}_{n}\in Y_{1}^{*}.$
\begin{theorem}  \label{palais }
The functional  $\displaystyle \psi$ satisfies the Palais-Smale condition at level $\displaystyle c \in\mathbb{R}.$
\end{theorem}
\begin{proof} 
Let $\displaystyle \{\mathrm{w}_{n, k}\}_{k\in \mathbb{N}} \subset Y_{1} $  be a Palais-Smale  sequence of $\psi$  at level c. Then, we have 
\begin{align}\label{smallcondition}
\psi\left(\mathrm{w}_{n, k}\right)=c+o(1),  \text { and }\psi^{\prime}\left(\mathrm{w}_{n, k}\right)=o(1).
\end{align}
\subsection*{Claim 1: The sequence  $\{\mathrm{w}_{n, k}\}_{k\in \mathbb{N}}$  is uniformly bounded in $Y_{1}$}
By using the contradiction approach, we prove Claim 1.
We assume the claim  1 does not hold, that is up to a subsequence still denoted by $\displaystyle \{\mathrm{w}_{n, k} \}_{k\in  \mathbb{N}}$ such that $\Vert \mathrm{w}_{n, k} \Vert_{Y_{1}}\to +\infty $ as $\displaystyle k\to +\infty $ in  $\displaystyle Y_{1}.$  Let  us $\displaystyle  \mathrm{v}_{n, k}:= \frac{\mathrm{w}_{n, k}}{ \|\mathrm{w}_{n, k} \|_{Y_{1}}}.$  Clearly $\displaystyle \{\mathrm{v}_{n, k}\}_{k\in  \mathbb{N}}$ is  bounded in $\displaystyle Y_{1}.$ Since  $Y_{1}$  is a reflexive Banach space, up to a subsequence still denoted by $\displaystyle \{\mathrm{v}_{n, k}\}_{k\in \mathbb{N}} $ such  that:
\begin{eqnarray}\label{reflexive}
\begin{gathered}
\left\{\begin{array}{llll}
   \mathrm{v}_{n, k} \rightharpoonup \mathrm{v}_{n}   \text { weakly in }  Y_{1} \text{ as }  k\to \infty, \\
\mathrm{v}_{n, k} \rightarrow \mathrm{v}_{n}   \text { strongly }   k\to +\infty  \text { in } L^{a(\mathrm{x})}(\mathcal{U}) \text { for  all } 1< a(\mathrm{x})< p_{s_{1}}^{*}(\mathrm{x}),\\ 
\mathrm{v}_{n, k} \to \mathrm{v}_{n}   \text { a.e  in }  \mathcal{U} \text{ as }  k\to \infty.
\end{array}\right.
\end{gathered}
\end{eqnarray}
Combining \eqref{smallcondition} with $ \frac{1}{\left\|\mathrm{w}_{n, k}\right\|_{Y_{1}}}=o(1), $ we have 
\begin{align}\label{k9}
\begin{split}
 &  \frac{\|\mathrm{v}_{n, k} \|^{p^{+}}_{Y_{1}}}{p^{-}}+ \frac{\|\mathrm{w}_{n, k} \|^{p^{+}-q^{-}}_{Y_{1}}  \|\mathrm{v}_{n, k} \|^{q^{+}}_{Y_{2}}}{q^{-}} -  \|\mathrm{w}_{n, k} \|^{-p^{-}}_{Y_{1}} \int_{\mathcal{U}} G_{n}(\mathrm{x}, \mathrm{w}_{n, k} ) d\mathrm{x}-  \frac{\|\mathrm{w}_{n, k} \|^{\sigma^{-}-p^{+}}_{Y_{1}}}{\sigma^{+}} \int_{\mathcal{U}}   \mathcal{V}(\mathrm{x}) \vert \mathrm{w}_{n}(\mathrm{x})+\frac{1}{n} \vert^{\sigma(\mathrm{x})} d \mathrm{x}\\
 &= o(1),
  \end{split}
  \end{align}
and
\begin{align}\label{kx}
\begin{split}
 & \|\mathrm{v}_{n, k} \|^{p^{+}}_{Y_{1}}+ \|\mathrm{w}_{n, k} \|^{q^{+}-p^{-}}_{Y_{1}}  \|\mathrm{v}_{n, k} \|^{q^{+}}_{Y_{2}}
  -\|\mathrm{w}_{n, k} \|^{-p^{-}}_{Y_{1}} \int_{\mathcal{U}}   \frac{g_{n}(\mathrm{x}, \mathrm{w}_{n, k} ) }{ {(\mathrm{w}_{n, k}+ \frac{1}{n})}^{\xi (\mathrm{x})} } \mathrm{w}_{n, k}(\mathrm{x}) d\mathrm{x}- \|\mathrm{w}_{n, k} \|^{\sigma^{-}-p^{+}}_{Y_{1}} \int_{\mathcal{U}}   \mathcal{V}(\mathrm{x}) \vert \mathrm{w}_{n}(\mathrm{x})+\frac{1}{n} \vert^{\sigma(\mathrm{x})} d \mathrm{x}\\
  &= o(1).
 \end{split}
  \end{align}
We use \eqref{kx} and \eqref{k9}, we have 
\begin{align}\label{k10}
\begin{split}
&\left( \frac{\eta}{p^{-}}-1\right) \|\mathrm{v}_{n, k} \|^{p^{-}}_{Y_{1}}+ \left( \frac{\eta}{q^{-}}-1\right) \|\mathrm{w}_{n, k} \|^{q^{-}-p^{+}}_{Y_{1}}  \|\mathrm{v}_{n, k} \|^{q^{+}}_{Y_{2}}- \left(\frac{\eta}{\sigma^{-}}-1  \right)  \|\mathrm{w}_{n, k} \|^{\sigma^{-}-p^{+}}_{Y_{1}} \int_{\mathcal{U}}   \mathcal{V}(\mathrm{x}) \vert \mathrm{w}_{n}(\mathrm{x})+\frac{1}{n} \vert^{\sigma(\mathrm{x})} d \mathrm{x}\\
 &- \eta\|\mathrm{w}_{n, k} \|^{-p^{-}}_{Y_{1}}  \int_{\mathcal{U}} \left(G_{n}(\mathrm{x}, \mathrm{w}_{n, k} )-\frac{g_{n}(\mathrm{x}, \mathrm{w}_{n, k} ) }{ {(\mathrm{w}_{n, k}+ \frac{1}{n})}^{\xi (\mathrm{x})} } \mathrm{w}_{n, k}(\mathrm{x}) \right) d\mathrm{x}= o(1).
\end{split}
  \end{align}
We use $( \mathcal{H}_{4}),$ we can write
 \begin{align}\label{k11}
\begin{split}
\left( \frac{\eta}{p^{-}}-1\right) \|\mathrm{v}_{n, k} \|^{p^{-}}_{Y_{1}}=&\left( 1- \frac{\eta}{q^{+}}\right)  \|\mathrm{w}_{n, k} \|^{q^{+}-p^{-}}_{Y_{1}}  \|\mathrm{v}_{n, k} \|^{q^{+}}_{Y_{2}}+  \eta \|\mathrm{w}_{n, k} \|^{-p^{-}}_{Y_{1}} \left( \int_{\mathcal{U}} G_{n}(\mathrm{x}, \mathrm{w}_{n, k} )-\frac{g_{n}(\mathrm{x}, \mathrm{w}_{n, k} ) }{ {(\mathrm{w}_{n, k}+ \frac{1}{n})}^{\xi (\mathrm{x})} } \mathrm{w}_{n, k}(\mathrm{x}) d\mathrm{x}\right) \\
&+  \left( 1- \frac{\eta}{\sigma^{+}}\right) \|\mathrm{w}_{n, k} \|^{\sigma^{-}-p^{+}}_{Y_{1}}\int_{\mathcal{U}}   \mathcal{V}(\mathrm{x}) \vert \mathrm{w}_{n}(\mathrm{x})+\frac{1}{n} \vert^{\sigma(\mathrm{x})} d \mathrm{x}+ o(1)\\
&\leq  \left( 1- \frac{\eta}{q^{+}}\right)  \|\mathrm{w}_{n, k} \|^{q^{+}-p^{-}}_{Y_{1}} \|\mathrm{v}_{n, k} \|^{q^{+}}_{Y_{2}}+  \left( 1- \frac{\eta}{\sigma^{+}}\right) \|\mathrm{w}_{n, k} \|^{\sigma^{-}-p^{+}}_{Y_{1}} \int_{\mathcal{U}}   \mathcal{V}(\mathrm{x}) \vert \mathrm{w}_{n}(\mathrm{x})+\frac{1}{n} \vert^{\sigma(\mathrm{x})} d \mathrm{x}\\
&+ a_{3}\|\mathrm{w}_{n, k}\|^{p^{-}- q^{-}}_{Y_{1}} \| \mathrm{v}_{n}\|^{p^{-}}_{L^{p^{-}}(\mathcal{U})} + o(1)\\
&= o(1).
\end{split}
 \end{align}
 as $k\to \infty.$ This is a contradiction as  $\displaystyle \| \mathrm{v}_{n}\|_{Y_{1}} =1,$  and hence Claim 1 follows. Consequently, there exists $\mathrm{w}_{n}\in Y_{1} $
 such that up to a subsequence
 \begin{eqnarray}\label{reflexive1}
\begin{gathered}
\left\{\begin{array}{llll}
   \mathrm{w}_{n, k} \rightharpoonup \mathrm{w}_{n}   \text { weakly in }  Y_{1} \text{ as }  k\to \infty, \\
\mathrm{w}_{n, k} \rightarrow \mathrm{w}_{n}   \text { strongly }    \text { in } L^{a(\mathrm{x})}(\mathcal{U}) \text { as } k\to +\infty  \text { for  all } 1< a(\mathrm{x})< p_{s_{1}}^{*}(\mathrm{x}),\\ 
\mathrm{w}_{n, k} \to \mathrm{w}_{n}   \text { a.e  in }  \mathcal{U} \text{ as }  k\to \infty.
\end{array}\right.
\end{gathered}
\end{eqnarray}
From $( \mathcal{H}_{1}),$  and $( \mathrm{V}),$  we get 
\begin{align} 
\int_{\mathcal{U}} \frac{ g_{n}\left(\mathrm{x}, \mathrm{w}_{n, k}\right) \mathrm{w}_{n, k} }{\frac{1}{n}+\mathrm{w}_{n, k} }d \mathrm{x} & =\int_{\mathcal{U}} \frac{g_{n}(\mathrm{x}, \mathrm{w}_{n} ) \mathrm{w}_{n} }{\frac{1}{n}+ \mathrm{w}_{n}} d \mathrm{x}+o(1), \\
\int_{\mathcal{U}} G_{n}\left(\mathrm{x}, \mathrm{w}_{n, k}\right) d \mathrm{x} & =\int_{\mathcal{U}} G_{n}(\mathrm{x},\mathrm{w}_{n} ) d \mathrm{x}+o(1),
\end{align}
 \begin{align} 
\int_{\mathcal{U}} \mathcal{V}(\mathrm{x}) \vert \mathrm{w}_{n, k}+\frac{1}{n}\vert^{\sigma(\mathrm{x})} d \mathrm{x}=\int_{\mathcal{U}} \mathcal{V}(\mathrm{x})\vert \mathrm{w}_{n}+\frac{1}{n} \vert^{\sigma(\mathrm{x})} d \mathrm{x}+o(1).
\end{align}
 We have also  $\displaystyle \{\mathrm{w}_{n, k}\}_{k\in \mathbb{N}}$ is bounded in  $\displaystyle Y_{2}.$ Since $\mathrm{w}_{n, k} \to \mathrm{w}_{n} $ a.e. in $\displaystyle \mathcal{U}$  as $k\to +\infty,$ we have that
\begin{align*}
\frac{\vert\mathrm{w}_{n, k}(\mathrm{x})-\mathrm{w}_{n, k}(\mathrm{y})\vert^{p(\mathrm{x}, \mathrm{y})-2}(\mathrm{w}_{n, k}(\mathrm{x})-\mathrm{w}_{n, k}(\mathrm{y}))}{ \vert \mathrm{x}-\mathrm{y}\vert^{(\frac{N}{p(\mathrm{x}, \mathrm{y})}+ s_{1})(p(\mathrm{x}, \mathrm{y})-1)}}
 \to  \frac{\vert\mathrm{w}_{n}(\mathrm{x})-\mathrm{w}_{n}(\mathrm{y})\vert^{p(\mathrm{x}, \mathrm{y})-2}(\mathrm{w}_{n}(\mathrm{x})-\mathrm{w}_{n}(\mathrm{y}))}{ \vert \mathrm{x}-\mathrm{y}\vert^{(\frac{N}{p(\mathrm{x}, \mathrm{y})}+ s_{1})(p(\mathrm{x}, \mathrm{y})-1)}}
\end{align*}
a.e $\displaystyle (\mathrm{x}, \mathrm{y})\in  \mathcal{U}\times \mathcal{U} $ as $\displaystyle k\to +\infty.$ Since  $\displaystyle \{ \mathrm{w}_{n, k} \}_{k\in\mathbb{N}}$ is bounded in $\displaystyle Y_{1},$  there exist  $\displaystyle c>0$ such that $$
\int_{\mathcal{U}\times \mathcal{U}} \vert \frac{\vert\mathrm{w}_{n, k}(\mathrm{x})-\mathrm{w}_{n, k}(\mathrm{y})\vert^{p(\mathrm{x}, \mathrm{y})-2}(\mathrm{w}_{n, k}(\mathrm{x})-\mathrm{w}_{n, k}(\mathrm{y}))}{ \vert \mathrm{x}-\mathrm{y}\vert^{(\frac{N}{p(\mathrm{x}, \mathrm{y})}+ s_{1})(p(\mathrm{x}, \mathrm{y})-1)}}\vert ^{ \frac{p(\mathrm{x}, \mathrm{y})}{p(\mathrm{x}, \mathrm{y})-1}}  d\mathrm{x}d\mathrm{y} \leq C.$$
So, we have that \begin{align*}
 \frac{\vert\mathrm{w}_{n, k}(\mathrm{x})-\mathrm{w}_{n, k}(\mathrm{y})\vert^{p(\mathrm{x}, \mathrm{y})-2}(\mathrm{w}_{n, k}(\mathrm{x})-\mathrm{w}_{n, k}(\mathrm{y}))}{ \vert \mathrm{x}-\mathrm{y}\vert^{(\frac{N}{p(\mathrm{x}, \mathrm{y})}+ s_{1})(p(\mathrm{x}, \mathrm{y})-1)}} 
   \rightharpoonup  \frac{\vert\mathrm{w}_{n}(\mathrm{x})-\mathrm{w}_{n}(\mathrm{y})\vert^{p(\mathrm{x}, \mathrm{y})-2}(\mathrm{w}_{n}(\mathrm{x})-\mathrm{w}_{n}(\mathrm{y}))}{ \vert \mathrm{x}-\mathrm{y}\vert^{(\frac{N}{p(\mathrm{x}, \mathrm{y})}+ s_{1})(p(\mathrm{x}, \mathrm{y})-1)}} \text{ as   } k\to \infty 
 \end{align*}
weakly in   $L^{p^{'}(\mathrm{x}, \mathrm{y})}(\mathcal{U}\times \mathcal{U}), $ where $\frac{1}{p^{'}(\mathrm{x}, \mathrm{y})}+ \frac{1}{p(\mathrm{x}, \mathrm{y})}=1.$ Let $\mathrm{w}_{n}\in Y_{1},  $ it is follows that  $$\frac{\mathrm{w}_{n}(\mathrm{x})-\mathrm{w}_{n}(\mathrm{y})}{\vert \mathrm{x}-\mathrm{y}\vert^{\frac{N}{p(\mathrm{x}, \mathrm{y})}+ s_{1}}}\in L^{p(\mathrm{x}, \mathrm{y})}(\mathcal{U}\times \mathcal{U}), \text { and }  \frac{\mathrm{w}_{n}(\mathrm{x})-\mathrm{w}_{n}(\mathrm{y})}{\vert \mathrm{x}-\mathrm{y}\vert^{\frac{N}{q(\mathrm{x}, \mathrm{y})}+ s_{2}}}\in L^{q(\mathrm{x}, \mathrm{y})}(\mathcal{U}\times \mathcal{U}).$$ 
Finally, we get that  
\begin{align*}
 &\int_{\mathcal{U}\times\mathcal{U}}\frac{\vert\mathrm{w}_{n, k}(\mathrm{x})-\mathrm{w}_{n, k}(\mathrm{y})\vert^{p(\mathrm{x}, \mathrm{y})-2}(\mathrm{w}_{n, k}(\mathrm{x})-\mathrm{w}_{n, k}(\mathrm{y}))}{ \vert \mathrm{x}-\mathrm{y}\vert^{(\frac{N}{p(\mathrm{x}, \mathrm{y})}+ s_{1})p(\mathrm{x}, \mathrm{y})}} d\mathrm{x}d\mathrm{y}  \to  \int_{\mathcal{U}\times\mathcal{U}} \frac{\vert\mathrm{w}_{n}(\mathrm{x})-\mathrm{w}_{n}(\mathrm{y})\vert^{p(\mathrm{x}, \mathrm{y})-2}(\mathrm{w}_{n}(\mathrm{x})-\mathrm{w}_{n}(\mathrm{y}))}{ \vert \mathrm{x}-\mathrm{y}\vert^{(\frac{N}{p(\mathrm{x}, \mathrm{y})}+ s_{1})p(\mathrm{x}, \mathrm{y})}} d\mathrm{x}d\mathrm{y}  \text  { as } k \to   \infty 
\end{align*}
and 
\begin{align*}
& \int_{\mathcal{U}\times\mathcal{U}}\frac{\vert\mathrm{w}_{n, k}(\mathrm{x})-\mathrm{w}_{n, k}(\mathrm{y})\vert^{q(\mathrm{x}, \mathrm{y})-2}(\mathrm{w}_{n, k}(\mathrm{x})-\mathrm{w}_{n, k}(\mathrm{y}))}{ \vert \mathrm{x}-\mathrm{y}\vert^{(\frac{N}{p(\mathrm{x}, \mathrm{y})}+ s_{2})q(\mathrm{x}, \mathrm{y})}} d\mathrm{x}d\mathrm{y} 
 \to  \int_{\mathcal{U}\times\mathcal{U}} \frac{\vert\mathrm{w}_{n}(\mathrm{x})-\mathrm{w}_{n}(\mathrm{y})\vert^{q(\mathrm{x}, \mathrm{y})-2}(\mathrm{w}_{n}(\mathrm{x})-\mathrm{w}_{n}(\mathrm{y}))}{ \vert \mathrm{x}-\mathrm{y}\vert^{(\frac{N}{q(\mathrm{x}, \mathrm{y})}+ s_{2})q(\mathrm{x}, \mathrm{y})}} d\mathrm{x}d\mathrm{y} \text  { as } k \to   \infty. 
\end{align*}
\subsection*{Claim 2: $\mathrm{w}_{n, k}\to  \mathrm{w}_{n}$ strongly in $Y_{1} $ as $k\to \infty.$}
Considering  the sequence defined as $\displaystyle \mathrm{v}_{n, k}= \mathrm{w}_{n, k} - \mathrm{w}_{n}.$ Since  $\displaystyle \mathrm{w}_{n, k} \to \mathrm{w}_{n}$ a.e in $\displaystyle \mathcal{U} $ and $\displaystyle \{\mathrm{w}_{n, k}\}_{k\in\mathbb{N}}$ is uniformly bounded in  $\displaystyle Y_{1}$ and $\displaystyle Y_{2}.$ Thanks to Brezis-Lieb Lemma in \cite{fu}, we have that  
 \begin{align}\label{k16}
\begin{split}\int_{\mathcal{U}\times\mathcal{U}}\frac{\vert\mathrm{w}_{n, k}(\mathrm{x})-\mathrm{w}_{n, k}(\mathrm{y})\vert^{p(\mathrm{x}, \mathrm{y})-2}(\mathrm{w}_{n, k}(\mathrm{x})-\mathrm{w}_{n, k}(\mathrm{y}))}{ \vert \mathrm{x}-\mathrm{y}\vert^{(\frac{N}{p(\mathrm{x}, \mathrm{y})}+ s_{1})p(\mathrm{x}, \mathrm{y})}} d\mathrm{x}d\mathrm{y}
 &=\int_{\mathcal{U}\times\mathcal{U}}\frac{\vert\mathrm{v}_{n, k}(\mathrm{x})-\mathrm{v}_{n, k}(\mathrm{y})\vert^{p(\mathrm{x}, \mathrm{y})-2}(\mathrm{v}_{n, k}(\mathrm{x})-\mathrm{v}_{n, k}(\mathrm{y}))}{ \vert \mathrm{x}-\mathrm{y}\vert^{(\frac{N}{p(\mathrm{x}, \mathrm{y})}+ s_{1})p(\mathrm{x}, \mathrm{y})}} d\mathrm{x}d\mathrm{y}\\
&+\int_{\mathcal{U}\times\mathcal{U}}\frac{\vert\mathrm{w}_{n}(\mathrm{x})-\mathrm{w}_{n}(\mathrm{y})\vert^{p(\mathrm{x}, \mathrm{y})-2}(\mathrm{w}_{n}(\mathrm{x})-\mathrm{w}_{n}(\mathrm{y}))}{ \vert \mathrm{x}-\mathrm{y}\vert^{(\frac{N}{p(\mathrm{x}, \mathrm{y})}+ s_{1})p(\mathrm{x}, \mathrm{y})}} d\mathrm{x}d\mathrm{y}+  o(1),
\end{split}
\end{align}
i.e $\displaystyle   \|\mathrm{w}_{n, k}\|^{p^{+}}_{Y_{1}}=  \|\mathrm{w}_{n}\|^{p^{+}}_{Y_{1}}+  \|\mathrm{v}_{n, k}\|^{p^{+}}_{Y_{1}} +o(1).$
  Similarly, we get 
$\displaystyle    \|\mathrm{w}_{n, k}\|^{q^{+}}_{Y_{2}}=  \|\mathrm{w}_{n}\|^{q^{+}}_{Y{2}}+\|\mathrm{v}_{n, k}\|^{q^{+}}_{Y_{2}} +o(1).$ So, we have that 
\begin{eqnarray}\label{k17}
\begin{gathered}
c+o(1)= \psi(\mathrm{w}_{n, k})
\leq \frac{1}{p^{+}}  \|\mathrm{v}_{n, k}\|^{p^{+}}_{Y_{1}}+ \frac{1}{q^{+}}  \|\mathrm{v}_{n, k}\|^{q^{+}}_{Y_{2}}
+ \frac{1}{p^{+}}  \|\mathrm{w}_{n}\|^{p^{+}}_{Y_{1}}+ \frac{1}{q^{+}}  \|\mathrm{w}_{n}\|^{q^{+}}_{Y_{2}}- \int_{\mathcal{U}}G_{n}(\mathrm{x}, \mathrm{w}_{n}(\mathrm{x}))d\mathrm{x}\\
- \int_{\mathcal{U}} \mathcal{V}(\mathrm{x}) \vert \mathrm{w}_{n, k}+\frac{1}{n}\vert^{\sigma(\mathrm{x})} d \mathrm{x}.
\end{gathered}
\end{eqnarray}
On the other hand, using $\displaystyle \psi^{'}(\mathrm{w}_{n, k}) \to 0 $ as $\displaystyle k\to +\infty,$ we have that 
\begin{eqnarray}\label{equation13}
\begin{gathered}
\lim_{k \to +\infty} \|\mathrm{v}_{n, k}\|^{p^{+}}_{Y_{1}}+ \|\mathrm{v}_{n, k}\|^{q^{+}}_{Y_{2}}=  \int_{\mathcal{U}} \frac{g_{n}(\mathrm{x}, \mathrm{w}_{n} ) }{ {(\mathrm{w}_{n}+ \frac{1}{n})}^{\xi (\mathrm{x})} } \mathrm{w}_{n}(\mathrm{x}) d\mathrm{x}
-\|\mathrm{w}_{n}\|^{p^{+}}_{Y_{1}}- \|\mathrm{w}_{n}\|^{q^{+}}_{Y_{2}}- \int_{\mathcal{U}} \mathcal{V}(\mathrm{x}) \vert \mathrm{w}_{n}+\frac{1}{n}\vert^{\sigma(\mathrm{x})} d \mathrm{x}.
\end{gathered}
\end{eqnarray}
We combine (\ref{equation13}) with $\displaystyle \psi(\mathrm{w}_{n})=0, $ we obtain that  $$\lim_{k \to +\infty} \|\mathrm{v}_{n, k}\|^{p^{+}}_{Y_{1}}+ \|\mathrm{v}_{n, k}\|^{q^{+}}_{Y_{2}}=0.$$
Since   $\displaystyle \|\mathrm{v}_{n, k}\|^{p^{+}}_{Y_{1}}$ and $\|\mathrm{v}_{n, k}\|^{q^{+}}_{Y_{2}}$ are bounded sequence, we can write $\displaystyle \lim_{k \to +\infty} \|\mathrm{v}_{n, k}\|^{p^{+}}_{Y_{1}}=a $ and  $\displaystyle \lim_{k \to +\infty} \|\mathrm{v}_{n, k}\|^{q^{+}}_{Y_{2}}=b.$ Since $\displaystyle a, b\geq 0 $ and  $\displaystyle a+b=0, $ we get that $\displaystyle  a=b=0.$
Finally $\displaystyle \mathrm{w}_{n, k}  \to \mathrm{w}_{n}$ strongly in $\displaystyle Y_{1}$ as $\displaystyle k\to +\infty.$
\end{proof}
Now, we will use the notion of the local $(m,n)$ linking for computing $\dim C_{k}(\psi, 0).$  
\begin{theorem}
The functional $\displaystyle \psi$ has a local $\displaystyle (1,1)-$ linking at the origin.
\end{theorem}
\begin{proof}
According to $(\mathcal{H}_{3})$ and a direct computation, we have 
\begin{align}\label{jjj}
\frac{ n^{\xi(\mathrm{x})} l}{2 p(\mathrm{x}, \mathrm{y})}  \vert \mathrm{w}_{n}(\mathrm{x}) \vert ^{p^{+}+1} \leq G_{n}(\mathrm{x}, \mathrm{w}_{n}(\mathrm{x})).
\end{align}
We define $\displaystyle V=\mathbb{R}.$  Clearly $\displaystyle V$ is a one dimensional vector space subspace of $\displaystyle Y_{1}.$ We choose $\displaystyle r\in(0,1)$ such that $\displaystyle K_{\psi}\cap \overline{B_{r}(0)}=\{0 \},$ where  $\displaystyle B_{r}(0)=\{ \mathrm{w}_{n}\in  Y_{1}: \|\mathrm{w}_{n}\|_{Y_{1}}<r \}$ and 
$\displaystyle K_{\psi}=\{ \mathrm{w}_{n}\in  Y_{1}:\psi^{'}(\mathrm{w}_{n})=0\}.$ We consider the set $\displaystyle E=V\cap \overline{B_{r}(0)}$ for small enough $r\in(0, 1).$ Recall that on a finite-dimensional normed space, all norms are equivalent. So, by taking $\displaystyle r\in (0,1)$ even Smaler as necessary, we obtain that $$\|\mathrm{w}_{n}\|_{Y_{1}} \leq r   \Rightarrow   \vert\mathrm{w}_{n}\vert  \leq \delta \quad  \text { for all } \quad \mathrm{w}_{n} \in V=\mathbb{R}.$$ Then for any $\displaystyle \mathrm{w}_{n} \in V \cap  \overline{B_{r}(0)},$ we have 
\begin{align*}
 \displaystyle \psi (\mathrm{w}_{n})&=\int_{\mathcal{U}\times \mathcal{U}}   \frac{1}{p(\mathrm{x}, \mathrm{y})} \frac{\vert\mathrm{w}_{n}(\mathrm{x})-\mathrm{w}_{n}(\mathrm{y})\vert^{p(\mathrm{x}, \mathrm{y})}}{ \vert \mathrm{x}-\mathrm{y}\vert^{N+s_{1} p(\mathrm{x}, \mathrm{y})} }d\mathrm{x}d\mathrm{y}+ \int_{\mathcal{U} \times \mathcal{U}}  \frac{1}{q(\mathrm{x}, \mathrm{y})} \frac{\vert\mathrm{w}_{n}(\mathrm{x})-\mathrm{w}_{n}(\mathrm{y})\vert^{q(\mathrm{x}, \mathrm{y})}}{ \vert \mathrm{x}-\mathrm{y}\vert^{N+s_{2} q(\mathrm{x}, \mathrm{y})} }d\mathrm{x}d\mathrm{y}\\
 &-\int_{\mathcal{U}} G_{n}(\mathrm{x}, \mathrm{w}_{n}(\mathrm{x}))d\mathrm{x}- \int_{ \mathcal{U}} \frac{ \mathcal{V}(\mathrm{x}) }{\sigma(\mathrm{x})}\vert \mathrm{w}_{n}(\mathrm{x})+\frac{1}{n} \vert^{\sigma(\mathrm{x})} d \mathrm{x}\\
&\leq   \frac{1}{p^{-}} \int_{\mathcal{U}\times \mathcal{U}} \frac{\vert\mathrm{w}_{n}(\mathrm{x})-\mathrm{w}_{n}(\mathrm{y})\vert^{p(\mathrm{x}, \mathrm{y})}}{ \vert \mathrm{x}-\mathrm{y}\vert^{N+s_{1} p(\mathrm{x}, \mathrm{y})} }d\mathrm{x}d\mathrm{y} 
+\frac{1}{q^{-}}  \int_{\mathcal{U}\times \mathcal{U}} \frac{\vert\mathrm{w}_{n}(\mathrm{x})-\mathrm{w}_{n}(\mathrm{y})\vert^{q(\mathrm{x}, \mathrm{y})}}{ \vert \mathrm{x}-\mathrm{y}\vert^{N+s_{2} q(\mathrm{x}, \mathrm{y})} }d\mathrm{x}d\mathrm{y}\\
&- \frac{ n ^{\xi^{+}}l}{2(p^{+}+1)}   \int_{\mathcal{U}} \vert \mathrm{w}_{n}(\mathrm{x})\vert^{p^{+}+1} d\mathrm{x}-\frac{1}{\sigma^{+}}  \left( \frac{1}{n}\right) ^{\sigma^{+}}\vert \mathcal{U}\vert\\  
&\leq 0.
\end{align*}
Further, we consider the set $$ \displaystyle D=\left\lbrace  \mathrm{w}_{n} \in  Y_{1}:  \min (\frac{1}{q^{+}}, \frac{1}{p^{+}})  \|\mathrm{w}_{n}\|^{p(\mathrm{x}, \mathrm{y})}_{Y_{1}} >  \|\beta\|_{\infty}  C(\mathcal{U}, r, N)\|\mathrm{w}_{n}\|^{l(\mathrm{x})}_{Y_{1}}\frac{n^{1-\xi^{-}}}{1-\xi^{-}} +\frac{1}{\sigma^{+}} \eta_{1}   \|\mathrm{w}_{n}\|_{Y_{1}}\right\rbrace,$$
where $l: \mathcal{U}\to (1, \infty)$ is the continuous function  such that $l(\mathrm{x})\leq p^{*}_{s}(\mathrm{x}),$ and $C(\mathcal{U}, r, N)$ is  the positive constant.
Using condition $( \mathcal{H}_{1}),$ $( \mathrm{V}),$   and  Theorem \ref{embedding}, we have that for any $\displaystyle \mathrm{w}_{n} \in  D,$  
\begin{align*}
 \displaystyle \psi (\mathrm{w}_{n})&= \displaystyle \int_{\mathcal{U}\times \mathcal{U}}   \frac{1}{p(\mathrm{x}, \mathrm{y})} \frac{\vert\mathrm{w}_{n}(\mathrm{x})-\mathrm{w}_{n}(\mathrm{y})\vert^{p(\mathrm{x}, \mathrm{y})}}{ \vert \mathrm{x}-\mathrm{y}\vert^{N+s_{1} p(\mathrm{x}, \mathrm{y})} }d\mathrm{x}d\mathrm{y}+ \int_{\mathcal{U} \times \mathcal{U}}  \frac{1}{q(\mathrm{x}, \mathrm{y})} \frac{\vert\mathrm{w}_{n}(\mathrm{x})-\mathrm{w}_{n}(\mathrm{y})\vert^{q(\mathrm{x}, \mathrm{y})}}{ \vert \mathrm{x}-\mathrm{y}\vert^{N+s_{2} q(\mathrm{x}, \mathrm{y})} }d\mathrm{x}d\mathrm{y}\\
 &-\int_{\mathcal{U}} G_{n}(\mathrm{x}, \mathrm{w}_{n}(\mathrm{x}))d\mathrm{x} -\int_{\mathcal{U} } \frac{ \mathcal{V}(\mathrm{x}) }{\sigma(\mathrm{x})}\vert \mathrm{w}_{n}(\mathrm{x})+\frac{1}{n} \vert^{\sigma(\mathrm{x})} d \mathrm{x}\\
& \displaystyle \geq  \frac{1}{p^{+}}\|\mathrm{w}_{n}\|^{p(\mathrm{x}, \mathrm{y})}_{Y_{1}}+ \frac{1}{q^{+}} \|\mathrm{w}_{n}\|^{q(\mathrm{x}, \mathrm{y})}_{Y_{2}}-   2\|\beta\|_{\infty}  \|\mathrm{w}_{n}\|^{l(\mathrm{x})}_{Y_{1}} -\frac{1}{\sigma^{+}} \eta _{1}  \|\mathrm{w}_{n}\|_{Y_{1}}
>0.
\end{align*}
Let $\displaystyle U=  \displaystyle  \overline{B_{r}(0)}, $ $E_{0}= \displaystyle V\cap   \partial B_{r}(0),$ $E=V\cap  \overline{B_{r}(0)}, $ and $\displaystyle D$ as above, we have that $\displaystyle 0\notin E_{0}\subset E \subset U=   \overline{B_{r}(0)}$ and  $\displaystyle E_{0}\cap  D=\emptyset. $ Therefore, we arrive the following  $$\psi_{\vert E} \leq 0 < \psi_{\vert D\cap \overline{B_{r}(0)}}.$$
Let $\displaystyle Y$  be the topological complement of $\displaystyle V$. We have that $\displaystyle Y_{1}= V\bigoplus Y.$ So, every $\displaystyle \mathrm{w}_{n}\in Y_{1} $ can be written in unique way as $$\mathrm{w}_{n}=\mathrm{v}_{n}+\mathrm{y}_{n} \,\,\, \text { with }\,\,\,  \mathrm{v}_{n}\in V,  \mathrm{y}_{n}\in Y.$$
We consider the map $\displaystyle h: [0, 1]\times Y_{1}\backslash D  \to  Y_{1}\backslash D $ defined by $$h(t, \mathrm{w}_{n})= (1-t)\mathrm{w}_{n}+ tr\frac{\mathrm{v}_{n}}{\|\mathrm{v}_{n}\|}. $$
We have $\displaystyle h(0, \mathrm{w}_{n})=\mathrm{w}_{n} $ and $\displaystyle h(1, \mathrm{w}_{n})= r \frac{\mathrm{v}_{n}}{\|\mathrm{v}_{n}\|} \in V \cap \partial  B_{r}(0)= E_{0}.$ It follows that $\displaystyle E_{0}$ is a deformation retract of $\displaystyle Y_{1}\backslash D.$ Hence  $$ i^{*}:H_{0}(E_{0}) \to  H_{0}(Y_{1})$$
$\displaystyle i$ an isomorphism. Note that $\displaystyle E_{0}=\{a, -a\}$ for some $\displaystyle a\neq0. $ Therefore, from $\displaystyle \dim  H_{0}(E_{0})=2, $ since $\displaystyle H_{0}(E_{0})= \mathbb{R}\bigoplus\mathbb{R}. $ Thus $\displaystyle \dim im(i^{*})=2. $\\
The set $\displaystyle E=V\cap B_{r}(0) $ is contractible (it is an interval ). By Theorem 11.5  in \cite{eilenberg}, we have that $\displaystyle H_{0}(E, E_{0})=0.$  Thanks to Remark 6.1.26  in \cite{eilenberg}, we get $\displaystyle \dim im(j^{*})=1.$ So, finally  $$\dim im(i^{*})-\dim im(j^{*})=2-1=1.$$ Thus the hypothesis of Definition \ref{definition} are satisfied.  Hence $\displaystyle \psi$ has a local $\displaystyle (1,1)-$  linking at $\displaystyle 0.$
\end{proof}
\begin{remark} 
For all $k\in\mathbb{N}, $ $C_{k}(\psi, 0)\neq0.$
\end{remark} 
\begin{proof}
Since $\psi$ has a local $(1,1)$ linking at the origin. By  proposition 2.1 in \cite{su}, we get that  $\dim C_{k}(\psi, 0)\geq 1.$
\end{proof}
Now, we will compute the group critical of $\psi$ at infinitely.
\begin{theorem}
Suppose that   the condition  $( \mathcal{H}_{3})$ is satisfied. 
Then, there   exists $\displaystyle k\in \mathbb{N}$ such that $\displaystyle C_{k}(\psi, \infty)=0. $ 
\end{theorem}
\begin{proof}
Firstly, we prove that there exists a positive constant A such that $\psi^{a}$  is homotopic to 
 exists a constant $\displaystyle A> 0$ such that $\psi^{a}$ is homotopic to $S^{1}=\{\mathrm{w}_{n} \in  Y_{1}:  \|\mathrm{w}_{n} \|_{Y_{1}}=1\}, $ for all $\displaystyle a<-A.$
From the condition $( \mathcal{H}_{3}), $  it follows that   
\begin{align*}
\psi(t\mathrm{w}_{n})&= \int_{\mathcal{U}\times \mathcal{U}} \frac{t^{p(\mathrm{x}, \mathrm{y})}}{p(\mathrm{x}, \mathrm{y})} \frac{\vert\mathrm{w}_{n}(\mathrm{x})-\mathrm{w}_{n}(\mathrm{y})\vert^{p(\mathrm{x}, \mathrm{y})}}{ \vert \mathrm{x}-\mathrm{y}\vert^{N+s_{1} p(\mathrm{x}, \mathrm{y})} }d\mathrm{x}d\mathrm{y}
+\int_{\mathcal{U} \times \mathcal{U}}  \frac{t^{q(\mathrm{x}, \mathrm{y})}}{q(\mathrm{x}, \mathrm{y})} \frac{\vert\mathrm{w}_{n}(\mathrm{x})-\mathrm{w}_{n}(\mathrm{y})\vert^{q(\mathrm{x}, \mathrm{y})}}{ \vert \mathrm{x}-\mathrm{y}\vert^{N+s_{2} q(\mathrm{x}, \mathrm{y})} }d\mathrm{x}d\mathrm{y}\\
&-\int_{\mathcal{U} } \frac{ \mathcal{V}(\mathrm{x}) }{\sigma(\mathrm{x})} t^{\sigma(\mathrm{x})}\vert \mathrm{w}_{n}(\mathrm{x})+\frac{1}{n} \vert^{\sigma(\mathrm{x})} d \mathrm{x} -\int_{\mathcal{U}} G_{n}(\mathrm{x}, t\mathrm{w}_{n}(\mathrm{x}))d\mathrm{x}\\
& \leq \frac{t^{p^{+}}}{p^{-}} + \frac{ t^{q^{+}}}{q^{-}}  \|\mathrm{w}_{n} \|^{q(\mathrm{x},  \mathrm{y})}_{Y_{2}}- \frac{l t^{p_{s_{1}}^{+*}}}{2 p_{s_{1}}^{+*}}  \int_{\mathcal{U}} \mathrm{w}_{n}(\mathrm{x})^{p_{s_{1}}^{+*}}d\mathrm{x}-\frac{\theta_{1}t^{\sigma^{+}}}{\sigma^{+}}\int_{\mathcal{U}}  \vert \mathrm{w}_{n}(\mathrm{x})+ \frac{1}{n}\vert^{\sigma(\mathrm{x})} d\mathrm{x}, 
\end{align*}
where $\displaystyle p_{s_{1}}^{+*}=\max _{\mathrm{x}\in \mathcal{U}  }p_{s_{1}}^{*}(\mathrm{x}). $ Since $\displaystyle p_{s_{1}}^{+*} >  p^{+}>  q^{+}>\sigma^{+}, $ we have that  $\displaystyle \psi(t \mathrm{w}_{n} ) \to -\infty $ as  $\displaystyle t\to  +\infty. $ Let $\displaystyle  A\in \mathbb{R} $ there   exists $t \in  \mathbb{R}  $ such that $\displaystyle  \|t \mathrm{w}_{n} \|_{Y_{1}}\geq B, $   we have that $\displaystyle \psi(t\mathrm{w}_{n}) \leq A.$ 
Since  $\displaystyle \mathrm{w}_{n} \in S^{1},$ we have that  
\begin{align*}
\displaystyle\frac{d}{dt} \psi(t\mathrm{w}_{n})&= \int_{\mathcal{U}\times \mathcal{U}} t^{p(\mathrm{x}, \mathrm{y})-1} \frac{\vert\mathrm{w}_{n}(\mathrm{x})-\mathrm{w}_{n}(\mathrm{y})\vert^{p(\mathrm{x}, \mathrm{y})}}{ \vert \mathrm{x}-\mathrm{y}\vert^{N+s_{1} p(\mathrm{x}, \mathrm{y})} }d\mathrm{x}d\mathrm{y}-
\displaystyle \int_{\mathcal{U}}  \mathrm{w}_{n}(\mathrm{x})\frac{g_{n}(\mathrm{x}, t\mathrm{w}_{n}(\mathrm{x}))}{(t \mathrm{w}_{n}(\mathrm{x})+ \frac{1}{n})^{\xi (\mathrm{x})}}d\mathrm{x}\\
&+\displaystyle \int_{\mathcal{U} \times \mathcal{U}}  t^{q(\mathrm{x}, \mathrm{y})-1} \frac{\vert\mathrm{w}_{n}(\mathrm{x})-\mathrm{w}_{n}(\mathrm{y})\vert^{q(\mathrm{x}, \mathrm{y})}}{ \vert \mathrm{x}-\mathrm{y}\vert^{N+s_{2} q(\mathrm{x}, \mathrm{y})} }d\mathrm{x}d\mathrm{y}-
 \int_{\mathcal{U}}  \mathcal{V}(\mathrm{x}) t^{\sigma(\mathrm{x})-1}\vert \mathrm{w}_{n}(\mathrm{x})+ \frac{1}{n}\vert^{\sigma(\mathrm{x})-1} d\mathrm{x}\\
& \displaystyle \leq t^{p^{+}-1}+ t^{q^{+}-1}  \|\mathrm{w}_{n} \|^{q(\mathrm{x},  \mathrm{y})}_{Y_{2}} -\int_{\mathcal{U}}  \mathrm{w}_{n}(\mathrm{x})\frac{g_{n}(\mathrm{x}, t\mathrm{w}_{n}(\mathrm{x}))}{(t \mathrm{w}_{n}(\mathrm{x})+ \frac{1}{n})^{\xi (\mathrm{x})}}d\mathrm{x}-  t^{\sigma^{-}-1}\int_{\mathcal{U}}  \mathcal{V}(\mathrm{x}) \vert \mathrm{w}_{n}(\mathrm{x})+ \frac{1}{n}\vert^{\sigma(\mathrm{x})-1} d\mathrm{x}\\
&\leq \frac{p^{+}}{t}\left[ A+  \int_{\mathcal{U}} G_{n}(\mathrm{x}, t\mathrm{w}_{n}(\mathrm{x}))d\mathrm{x}-  \frac{p^{+}}{t} \int_{\mathcal{U}}  t \mathrm{w}_{n}(\mathrm{x})\frac{g_{n}(\mathrm{x}, t\mathrm{w}_{n}(\mathrm{x}))}{(t \mathrm{w}_{n}(\mathrm{x})+ \frac{1}{n})^{\xi (\mathrm{x})}}d\mathrm{x}\right] - t^{\sigma^{-}-1}\int_{\mathcal{U}}  \mathcal{V}(\mathrm{x}) \vert \mathrm{w}_{n}(\mathrm{x})+ \frac{1}{n}\vert^{\sigma(\mathrm{x})-1} d\mathrm{x}\\
&\displaystyle \leq \frac{p^{+}}{t}\left[ A+    \left( \frac{1}{\theta }-  \frac{1}{ p^{+}}\right)  \int_{\mathcal{U}} t  \mathrm{w}_{n}(\mathrm{x})  \frac{g_{n}(\mathrm{x}, t\mathrm{w}_{n}(\mathrm{x}))}{(t \mathrm{w}_{n}(\mathrm{x})+ \frac{1}{n})^{\xi (\mathrm{x})}}d\mathrm{x}\right]  \\
& \displaystyle \leq   \frac{p^{+}}{t}\left[ A+     C_{1}\left( \frac{1}{\theta }-  \frac{1}{ p^{+}}\right) \right]- t^{\sigma^{-}-1}\int_{\mathcal{U}}  \mathcal{V}(\mathrm{x}) \vert \mathrm{w}_{n}(\mathrm{x})+ \frac{1}{n}\vert^{\sigma(\mathrm{x})-1} d\mathrm{x}\\
& <0.
\end{align*}
By the implicit function Theorem, there exists an unique  $\displaystyle T\in C(S^{1}, \mathbb{R})$ such that for any $\displaystyle \mathrm{w}_{n}\in  S^{1},$ $$ \psi(T(\mathrm{w}_{n} )\mathrm{w}_{n})=A.$$
For any $\displaystyle \mathrm{w}_{n}\neq 0,$ set $\displaystyle \tau (\mathrm{w}_{n})=\frac{1}{\|\mathrm{w}_{n}\|} T(\frac{\mathrm{w}_{n} }{\|\mathrm{w}_{n} \|}).$ Then $\displaystyle \tau \in  C(Y_{1} \backslash 0, \mathbb{R})$ and for all $\displaystyle \mathrm{w}_{n}\in  Y_{1}, $  $\displaystyle \psi( \mathrm{w}_{n} \tau(\mathrm{w}_{n}))=A.$ Moreover, if $\displaystyle \psi(\mathrm{w}_{n})=A, $ then $\displaystyle \tau(\mathrm{w}_{n})=1.$
We define a function  $\displaystyle \tau_{1}:  Y_{1} \to \mathbb{R}$ as   \\
 $$\tau_{1}(\mathrm{w}_{n}):=
\left\{\begin{array}{ll}
 \tau(\mathrm{w}_{n}), &   \text{ if }  \psi(\mathrm{w}_{n})\geq A, \\
 1, & \text { if } \psi(\mathrm{w}_{n})< A.
\end{array}
\right.\\  
$$     
Since   $ \psi(\mathrm{w}_{n})=A$ implies that $\displaystyle \tau (\mathrm{w}_{n})=1, $ we conclude that $\displaystyle \tau_{1} \in  C(Y_{1} \backslash 0, \mathbb{R}).$
Finally, we set $\displaystyle H:[0,1]\times  Y_{1} \backslash 0 \to Y_{1} \backslash 0$ as 
$$
H(t, \mathrm{w}_{n})=(1-t) \mathrm{w}_{n}+t \tau_{1}(\mathrm{w}_{n}) \mathrm{w}_{n}.
$$
We have $\displaystyle H(0, \mathrm{w}_{n})=\mathrm{w}_{n},$ $\displaystyle  H(1, \mathrm{w}_{n})=\tau_{1}(\mathrm{w}_{n})\mathrm{w}_{n}\in \psi^{A}, $ and $\displaystyle H(t, .)_{\vert\psi^{A}}=id_{\vert\psi^{A}}$ for all $\displaystyle t\in[0,1].$ It follows that  \begin{equation} \label{homotopy}
\psi^{A} \text {  is a strong deformation retract of   }  Y_{1}\backslash 0.
\end{equation}
We consider the radial retraction $\displaystyle r: Y_{1}\to \mathbb{R}$ defined by  $$r(\mathrm{w}_{n})=\frac{\mathrm{w}_{n}}{\|\mathrm{w}_{n}\|} \text {  for all } \mathrm{w}_{n} \in  Y_{1}. $$
This map is continuous and  $\displaystyle r_{\vert S^{1}}=id_{\vert S^{1}}.$ Therefore,  $\displaystyle S^{1}$ is a retract of  $\displaystyle Y_{1}\backslash 0.$ 
 Considering the map defined by $$h(t, \mathrm{w}_{n})=(1-t)\mathrm{w}_{n}+ tr(\mathrm{w}_{n}) \text {   for all }  (t, \mathrm{w}_{n})\in [0,1]\times  Y_{1}\backslash 0.$$
Then, $\displaystyle h(0, \mathrm{w}_{n})=\mathrm{w}_{n},$ $h(1, \mathrm{w}_{n})= r(\mathrm{w}_{n})\in S^{1},$ and $\displaystyle h(1, .)_{\vert S^{1}}= id _{\vert S^{1}}.$ Hence, we refer that 
\begin{equation} \label{retract}
S^{1} \text {  is a deformation retract of }  Y_{1}\backslash 0.
\end{equation}
Finally, by \ref{homotopy} and \ref{retract} it follow that $\displaystyle \psi^{a}$ and $S^{1}$ are homotopie equivalent. We already know that the space $Y_{1}$ is an infinite dimensional Banach space. From Remark 6.1.13  in \cite{papageorgiou}, it follows that the sphere unit $S^{1}$ is contractible. So, we have that  $$\displaystyle H_{k}(Y_{1}, \psi^{a}) = H_{k}(Y_{1}, S^{1})=0   \text { for all },\displaystyle k\in \mathbb{N}.$$
Finally, we obtain that \begin{align}\label{contractible}
\displaystyle C_{k}(\psi, \infty)= \displaystyle  H_{k}(Y_{1}, \psi^{a})
= \displaystyle H_{k}(Y_{1}, S^{1})
= \displaystyle 0,  \text { for all }  k\in \mathbb{N}.
\end{align}
\end{proof}
\begin{theorem}
Suppose that  conditions $( \mathrm{V}),$ and  $( \mathcal{H}_{1})- ( \mathcal{H}_{4}) $  are satisfied. Then, the problem (\ref{pro}) has  nontrivial weak solution in $\displaystyle Y_{1}.$
\end{theorem}
\begin{proof}
Since $\displaystyle \psi$ has a  local $\displaystyle (1, 1)-$ linking near the origin, then $\displaystyle \dim C_{1}(\psi, 0)\geq 1,$  i,e  $\displaystyle  C_{k}(\psi, 0) \neq 0\ $ for some $\displaystyle k\in\mathbb{N}.$ Thanks to  Theorem 6.2.42 in \cite{papageorgiou}, there exists $\displaystyle \mathrm{w}_{n}\in K_{\psi}.$
\end{proof}
\begin{theorem}
Suppose that  condition $( \mathrm{V}),$ and  $( \mathcal{H}_{1})- ( \mathcal{H}_{4}) $  are satisfied. Then, the problem (\ref{pro})  has at least non-trivial weak solution in $\displaystyle Y_{1}.$
\end{theorem}
\begin{proof}
Thanks to Theorem  \ref{palais } $\displaystyle \psi$ satisfies the Palais-Smale condition and is bounded from below and the trivial solution $\mathrm{w}_{n}=0$ is homological nontrivial and is it a minimizer. The conclusion follows from Theorem 2.1 in \cite{ su}.  
\end{proof}
\subsection{Existence of infinitely non-trivial solutions}\label{infinitly}
\begin{theorem}
Suppose that conditions $( \mathrm{V}),$ and  $( \mathcal{H}_{1})- ( \mathcal{H}_{4}) $   are satisfied. Then, the problem (\ref{pro})  has infinitely non-trivial weak solutions in $\displaystyle Y_{1}.$
\end{theorem}
\begin{proof}
 We suppose that  our problem admits  three non-trivial solutions
 $\displaystyle Y_{1}.$ That is $\displaystyle K_{\psi}= \{ 0, \mathrm{w}_{n}, \mathrm{v}_{n}\}.$ From Morse's relation, it follows that     \\
 $$C_{n}(\psi, 0)=
\left\{\begin{array}{ll}
 \mathbb{R}, & k=m(0), \\
 0, & \text { otherwise },
\end{array}
\right.\\  
$$     
where  $m(0)$  is a Morse index of $0$. We use  Morse's relation, we get that 
\begin{align*}
\displaystyle \sum_{k\geq 0} \rank C_{k}(\psi, \infty)X^{k}+ (1+X)Q(X)
=& \displaystyle  \sum_{k\geq 0} \rank C_{k}(\psi, 0)X^{k}+  \sum_{k\geq 0} \rank C_{k}(\psi, \mathrm{w}_{n})X^{k}+ \sum_{k\geq 0} \rank C_{k}(\psi, \mathrm{v}_{n})X^{k}&\\
=& \displaystyle X^{m(0)} + 2  \sum_{k\geq 0} \beta _{k} X^{k}.
\end{align*}
From  (\ref{contractible}), it follows that $$ \displaystyle (1+ X) Q(X)= X^{m(0)}+ 2 \sum_{k\geq 0} \beta _{k} X^{k},
$$
where $\displaystyle \beta _{k}$  nonnegative integer  and $\displaystyle Q$ is a polynomial with nonnegative integer coefficient. In particular, for $\displaystyle X=1$ we have  $\displaystyle 2a= 1+ 2 \sum_{k\geq 0} \beta_{k}.$ Since $ \displaystyle \beta_{k}\in \mathbb{N}, $ we have that  $\displaystyle \sum_{k\geq 0} \beta_{k}=+\infty $ leads to a contradiction. Thus, there exist infinitely solutions to the problem  (\ref{pro}). 
\end{proof}
\section{Fundamental Theorem}
\begin{theorem}\label{theorem fundamental}
Suppose that    conditions $( \mathrm{V}),$ and  $( \mathcal{H}_{1})- ( \mathcal{H}_{4}) $  are satisfied. Then, problem (\ref{k1}) admits an infinitely weak solutions in $\displaystyle Y_{1}.$
\end{theorem}
\begin{proof}
Let $ \displaystyle \{\mathrm{w}_{n}\}_{n\in\mathbb{N}}\subset Y_{1} $ be the sequence of solutions to problem (\ref{pro}). So, we have that 
\begin{align}\label{24}
\begin{split}
&\displaystyle \int_{\mathcal{U} \times \mathcal{U}} \frac{\vert\mathrm{w}_{n}(\mathrm{x})-\mathrm{w}_{n}(\mathrm{y})\vert^{p(\mathrm{x}, \mathrm{y})-2}(\mathrm{w}_{n}(\mathrm{x})-\mathrm{w}_{n}(\mathrm{y}))(\varphi(\mathrm{x})-\varphi(\mathrm{y}))}{ \vert \mathrm{x}-\mathrm{y}\vert^{N+s_{1} p(\mathrm{x}, \mathrm{y})} }d\mathrm{x}d\mathrm{y}\\ \displaystyle
& +\int_{\mathcal{U} \times \mathcal{U}} \frac{\vert\mathrm{w}_{n}(\mathrm{x})-\mathrm{w}_{n}(\mathrm{y})\vert^{q(\mathrm{x}, \mathrm{y})-2}(\mathrm{w}_{n}(\mathrm{x})-\mathrm{w}_{n}(\mathrm{y}))(\varphi(\mathrm{x})-\varphi(\mathrm{y}))}{ \vert \mathrm{x}-\mathrm{y}\vert^{N+s_{2} q(\mathrm{x}, \mathrm{y})} }d\mathrm{x}d\mathrm{y}\\ \displaystyle
 & =\int_{\mathcal{U}} \frac{g_{n}(\mathrm{x}, \mathrm{w}_{n}(\mathrm{x}))}{(\mathrm{w}_{n}(\mathrm{x})+\frac{1}{n})^{\xi (\mathrm{x})}}
  \varphi(\mathrm{x})d\mathrm{x}+ \int_{ \mathcal{U}}  \mathcal{V}(\mathrm{x})\vert \mathrm{w}_{n}(\mathrm{x})+\frac{1}{n} \vert^{\sigma(\mathrm{x})-2} ( \mathrm{w}_{n}(\mathrm{x})+\frac{1}{n} )  \varphi(\mathrm{x})  d \mathrm{x},  \text { for all }  \varphi\in Y_{1}^{*}.
\end{split}
\end{align}
We take  $ \displaystyle \varphi= \mathrm{w}_{n} $   in   (\ref{24}), we have that 
\begin{align*}
&\int_{\mathcal{U} \times \mathcal{U}} \frac{\vert\mathrm{w}_{n}(\mathrm{x})-\mathrm{w}_{n}(\mathrm{y})\vert^{p(\mathrm{x}, \mathrm{y})}}{ \vert \mathrm{x}-\mathrm{y}\vert^{N+s_{1} p(\mathrm{x}, \mathrm{y})} }d\mathrm{x}d\mathrm{y}+
 \int_{\mathcal{U} \times \mathcal{U}} \frac{\vert\mathrm{w}_{n}(\mathrm{x})-\mathrm{w}_{n}(\mathrm{y})\vert^{q(\mathrm{x}, \mathrm{y})}}{ \vert \mathrm{x}-\mathrm{y}\vert^{N+s_{2} q(\mathrm{x}, \mathrm{y})} }d\mathrm{x}d\mathrm{y}\\
 =& \displaystyle \int_{\mathcal{U}} \frac{g_{n}(\mathrm{x}, \mathrm{w}_{n}(\mathrm{x}))}{(\mathrm{w}_{n}(\mathrm{x})+\frac{1}{n})^{\xi (\mathrm{x})}}
\mathrm{w}_{n}(\mathrm{x})d\mathrm{x} +\int_{ \mathcal{U}}  \mathcal{V}(\mathrm{x})\vert \mathrm{w}_{n}(\mathrm{x}) +\frac{1}{n}\vert^{\sigma(\mathrm{x})-2} (\mathrm{w}_{n}(\mathrm{x}) +\frac{1}{n}) \mathrm{w}_{n}(\mathrm{x})   d \mathrm{x} \\
 \leq & \displaystyle \int_{\mathcal{U}} \frac{g_{n}(\mathrm{x}, \mathrm{w}_{n}(\mathrm{x}))}{(\mathrm{w}_{n}(\mathrm{x})+\frac{1}{n})^{\xi (\mathrm{x})}}
\mathrm{w}_{n}(\mathrm{x})d\mathrm{x} +\int_{ \mathcal{U}}  \mathcal{V}(\mathrm{x})\vert \mathrm{w}_{n}(\mathrm{x})+\frac{1}{n} \vert^{\sigma(\mathrm{x})}  d \mathrm{x}. 
 \end {align*} 
 
 Combining  $(\mathcal{H}_{1})$  with $(\mathrm{V}),$  it follows that 
\begin{align*}
&\int_{\mathcal{U} \times \mathcal{U}} \frac{\vert\mathrm{w}_{n}(\mathrm{x})-\mathrm{w}_{n}(\mathrm{y})\vert^{p(\mathrm{x}, \mathrm{y})}}{ \vert \mathrm{x}-\mathrm{y}\vert^{N+s_{1} p(\mathrm{x}, \mathrm{y})} }d\mathrm{x}d\mathrm{y}+
 \int_{\mathcal{U} \times \mathcal{U}} \frac{\vert\mathrm{w}_{n}(\mathrm{x})-\mathrm{w}_{n}(\mathrm{y})\vert^{q(\mathrm{x}, \mathrm{y})}}{ \vert \mathrm{x}-\mathrm{y}\vert^{N+s_{2} q(\mathrm{x}, \mathrm{y})} }d\mathrm{x}d\mathrm{y}\\
 &\leq  \displaystyle \int_{\mathcal{U}}\frac{ g_{n}(\mathrm{x}, \mathrm{w}_{n}(\mathrm{x}))}{(\mathrm{w}_{n}(\mathrm{x})+\frac{1}{n})^{\xi (\mathrm{x})}}
\mathrm{w}_{n}(\mathrm{x})d\mathrm{x} +\int_{ \mathcal{U}}  \mathcal{V}(\mathrm{x})\vert \mathrm{w}_{n}(\mathrm{x})+\frac{1}{n} \vert^{\sigma(\mathrm{x})}  d \mathrm{x}\\
&\leq \int_{\mathcal{U}} g(\mathrm{x}, \mathrm{w}_{n}(\mathrm{x}))  \vert\mathrm{w}_{n}(\mathrm{x}))\vert ^{1- \xi (\mathrm{x})} d\mathrm{x}+ \eta_{1} \Vert\mathrm{w}_{n} \Vert_{Y_{1}} \\
 &\leq   \int_{\mathcal{U}}  \beta(\mathrm{x})(1+   \vert\mathrm{w}_{n}(\mathrm{x}))\vert ^{r(\mathrm{x})-1) } )\vert\mathrm{w}_{n}(\mathrm{x})\vert ^{1- \xi (\mathrm{x})} d\mathrm{x}+ \eta_{1} \Vert\mathrm{w}_{n} \Vert_{Y_{1}}.
\end{align*}
Since $\displaystyle r(\mathrm{x})-1\leq  r(\mathrm{x})-\xi (\mathrm{x})$ for all $\displaystyle \mathrm{x}\in  \mathcal{U}, $ we get that 
$$  \displaystyle\|\mathrm{w}_{n} \|_{Y_{1}}\leq  \frac{ \|\beta  \|_{\infty} C(r(\mathrm{x}), p(\mathrm{x}, \mathrm{y}), q_{1}(\mathrm{x}), \xi(\mathrm{x}),  s_{1}, \mathcal{U})}{1-\eta_{1}}.$$
Hence,   the  sequence $\displaystyle \{\mathrm{w} _{n}\}_{n \in \mathbb{N}}$  is bounded  in $\displaystyle Y_{1}.$ Since  $\displaystyle Y_{1} $ is a reflexive Banach space, up to a subsequence, still denoted by  $\displaystyle \{\mathrm{w}_{n}\}$ such that  $\displaystyle \mathrm{w}_{n} \rightharpoonup  \mathrm{w}$ weakly in $Y_{1}, $  $\displaystyle \mathrm{w}_{n} \to   \mathrm{w} $ strongly in $\displaystyle L^{a(\mathrm{x})}(\mathcal{U})$ for $\displaystyle 1\leq a(\mathrm{x})<  p^{*}_{s_{1}}(\mathrm{x}), $ and $\displaystyle  \mathrm{w}_{n} \to   \mathrm{w} $ a.e in $\mathcal{U}.$ A similar discussion as in Theorem  \ref{palais } gives that
\begin{eqnarray}\label{k25}
\begin{gathered}
\displaystyle \lim_{n\to \infty} [\int_{\mathcal{U} \times \mathcal{U}} \frac{\vert\mathrm{w}_{n}(\mathrm{x})-\mathrm{w}_{n}(\mathrm{y})\vert^{p(\mathrm{x}, \mathrm{y})-2}(\mathrm{w}_{n}(\mathrm{x})-\mathrm{w}_{n}(\mathrm{y}))(\varphi(\mathrm{x})-\varphi(\mathrm{y}))}{ \vert \mathrm{x}-\mathrm{y}\vert^{N+s_{1} p(\mathrm{x}, \mathrm{y})} }d\mathrm{x}d\mathrm{y}\\ \displaystyle+ \int_{\mathcal{U} \times \mathcal{U}} \frac{\vert\mathrm{w}_{n}(\mathrm{x})-\mathrm{w}_{n}(\mathrm{y})\vert^{q(\mathrm{x}, \mathrm{y})-2}(\mathrm{w}_{n}(\mathrm{x})-\mathrm{w}_{n}(\mathrm{y}))(\varphi(\mathrm{x})-\varphi(\mathrm{y}))}{ \vert \mathrm{x}-\mathrm{y}\vert^{N+s_{2} q(\mathrm{x}, \mathrm{y})} }d\mathrm{x}d\mathrm{y}]\\
 = 
 \displaystyle \int_{\mathcal{U} \times \mathcal{U}} \frac{\vert\mathrm{w}(\mathrm{x})-\mathrm{w}(\mathrm{y})\vert^{p(\mathrm{x}, \mathrm{y})-2}(\mathrm{w}(\mathrm{x})-\mathrm{w}(\mathrm{y}))(\varphi(\mathrm{x})-\varphi(\mathrm{y}))}{ \vert \mathrm{x}-\mathrm{y}\vert^{N+s_{1} p(\mathrm{x}, \mathrm{y})} }d\mathrm{x}d\mathrm{y}\\
 +\int_{\mathcal{U} \times \mathcal{U}} \frac{\vert\mathrm{w}(\mathrm{x})-\mathrm{w}(\mathrm{y})\vert^{q(\mathrm{x}, \mathrm{y})-2}(\mathrm{w}(\mathrm{x})-\mathrm{w}(\mathrm{y}))(\varphi(\mathrm{x})-\varphi(\mathrm{y}))}{ \vert \mathrm{x}-\mathrm{y}\vert^{N+s_{2} q(\mathrm{x}, \mathrm{y})} }d\mathrm{x}d\mathrm{y}.
\end{gathered}
\end{eqnarray}
 Since $\displaystyle \mathrm{w}_{n}(\mathrm{x})>0, $ we get that  $$\displaystyle \vert\frac{g_{n}(\mathrm{x}, \mathrm{w}_{n}(\mathrm{x}))\varphi(\mathrm{x})}{ (\frac{1}{n}+ \mathrm{w}_{n}(\mathrm{x}))^{\xi(\mathrm{x})}}\vert \leq \vert g(\mathrm{x}, \mathrm{w}(\mathrm{x}))\varphi(\mathrm{x})\vert.$$
From  the dominated converge theorem, it follows that  $$\displaystyle \lim_{n\to + \infty}  \int_{\mathcal{U} } \frac{g_{n}(\mathrm{x}, \mathrm{w}_{n}(\mathrm{x}))\varphi(\mathrm{x})}{ (\frac{1}{n}+ \mathrm{w}_{n}(\mathrm{x}))^{\xi(\mathrm{x})}}d\mathrm{x} =\displaystyle \int_{\mathcal{U} } \frac{g(\mathrm{x}, \mathrm{w}(\mathrm{x}))\varphi(\mathrm{x})}{ ( \mathrm{w}(\mathrm{x}))^{\xi(\mathrm{x})}}d\mathrm{x}.$$
Similarly, we prove that  \begin{align}
 \lim_{n\to \infty}\int_{\mathcal{U}} \mathcal{V}(\mathrm{x}) \vert \mathrm{w}_{n}+\frac{1}{n}\vert^{\sigma(\mathrm{x})-2} (\mathrm{w}_{n}+\frac{1}{n} )\varphi(\mathrm{x}) d \mathrm{x}=\int_{\mathcal{U}} \mathcal{V}(\mathrm{x})\vert \mathrm{w} \vert^{\sigma(\mathrm{x})-2}  \mathrm{w} (\mathrm{x}) \varphi(\mathrm{x}) d \mathrm{x}.
\end{align}
 Finally, passing to the limit in \eqref{24}, we deduce that 
 \begin{eqnarray}
\begin{gathered}
\displaystyle \int_{\mathcal{U} \times \mathcal{U}} \frac{\vert\mathrm{w}(\mathrm{x})-\mathrm{w}(\mathrm{y})\vert^{p(\mathrm{x}, \mathrm{y})-2}(\mathrm{w}(\mathrm{x})-\mathrm{w}(\mathrm{y}))(\varphi(\mathrm{x})-\varphi(\mathrm{y}))}{ \vert \mathrm{x}-\mathrm{y}\vert^{N+s_{1} p(\mathrm{x}, \mathrm{y})} }d\mathrm{x}d\mathrm{y}
 + \int_{\mathcal{U} \times \mathcal{U}} \frac{\vert\mathrm{w}(\mathrm{x})-\mathrm{w}(\mathrm{y})\vert^{q(\mathrm{x}, \mathrm{y})-2}(\mathrm{w}(\mathrm{x})-\mathrm{w}(\mathrm{y}))(\varphi(\mathrm{x})-\varphi(\mathrm{y}))}{ \vert \mathrm{x}-\mathrm{y}\vert^{N+s_{2} q(\mathrm{x}, \mathrm{y})} }d\mathrm{x}d\mathrm{y}\\= \displaystyle
\int_{\mathcal{U} } \frac{g(\mathrm{x}, \mathrm{w}(\mathrm{x}))\varphi(\mathrm{x})}{ ( \mathrm{w}(\mathrm{x}))^{\xi(\mathrm{x})}}d\mathrm{x}+\int_{\mathcal{U}} \mathcal{V}(\mathrm{x})\vert \mathrm{w} \vert^{\sigma(\mathrm{x})-2}  \mathrm{w} (\mathrm{x}) \varphi(\mathrm{x})  d \mathrm{x},  \text { for all }  \varphi\in Y_{1}^{*},
\end{gathered}
\end{eqnarray}
namely $\displaystyle \mathrm{w}$ is a weak solution to \eqref{k1}.
\end{proof}
\section*{Declarations}

\begin{itemize}
\item \textbf{Ethics approval and consent to participate}\\
This article does not contain any studies with human participants performed by any of the authors.
\item \textbf{Funding }\\
Not applicable
\item \textbf{Consent for publication}\\
The authors  consent for publication.
\item \textbf{Conflict of interest}\\
The authors have no conflicts of interest to declare that are relevant to the content of this article.
\item \textbf{  Authors' contributions}\\
These authors contributed equally to this work.
\item \textbf{Availability of data and materials}\\
Not applicable.
\end{itemize}
\section*{Acknowledgments}
The authors would like to thank the referees for their suggestions and helpful comments which have improved the presentation of the original manuscript.

\end{document}